\documentclass[11pt]{amsart}
\calclayout

\usepackage{color,float}
\usepackage{hyperref, enumitem}
\usepackage{mathtools, comment}
\usepackage[capitalize,nameinlink,noabbrev,nosort]{cleveref}
\usepackage{graphicx,mathdots}
\usepackage{tikz,tikz-cd}

\usepackage[
  backend=biber,
  style=alphabetic,    
  giveninits=true,      
   url=false,
  doi=false,
  isbn=false,
  maxalphanames=99,
  maxbibnames=99       
]{biblatex}
\newtheorem{theorem}{Theorem}[section]
\newtheorem{lemma}[theorem]{Lemma}

\newtheorem{conjecture}[theorem]{Conjecture}
\crefname{section}{Section}{Sections}
\crefname{lemma}{Lemma}{Lemmas}
\crefname{theorem}{Theorem}{Theorems}

\theoremstyle{definition}
\newtheorem{defn}[theorem]{Definition}
\newtheorem{example}[theorem]{Example}

\theoremstyle{remark}
\newtheorem{remark}[theorem]{Remark}

\numberwithin{equation}{section}

\DeclareMathOperator{\Skip}{skip}
\DeclareMathOperator{\Free}{free}

\DeclareMathOperator{\qnona}{quinv-non-attacking}
\DeclareMathOperator{\isort}{inv-sorted}
\DeclareMathOperator{\qsort}{quinv-sorted}

\DeclareMathOperator{\perm}{perm}
\DeclareMathOperator{\proj}{\mathcal{P}}
\DeclareMathOperator{\MLQ}{MLQ}
\DeclareMathOperator{\dg}{dg}
\newcommand{\T}{\mathcal{T}}
\newcommand{\Z}{\mathbb{Z}}
\DeclareMathOperator{\rarm}{\widehat{arm}}
\DeclareMathOperator{\nontriv}{non-trivial}
\DeclareMathOperator{\triv}{trivial}

\DeclareMathOperator{\HH}{\widetilde{H}}
\DeclareMathOperator{\ASEP}{\mathcal{A}}

\DeclareMathOperator{\coquinv}{coquinv}

\DeclareMathOperator{\leg}{leg}
\DeclareMathOperator{\maj}{maj}
\DeclareMathOperator{\inv}{inv}
\DeclareMathOperator{\inc}{inc}

\DeclareMathOperator{\B}{\mathcal{B}}

\DeclareMathOperator{\Des}{Des}

\DeclareMathOperator{\wt}{wt}

\DeclareMathOperator{\quinv}{quinv}
\DeclareMathOperator{\TAZRP}{\mathcal{T}}
\DeclareMathOperator{\dep}{dep}

\DeclareMathOperator{\Tab}{Tab}
\DeclareMathOperator{\South}{South}

\newcommand{\qbinom}[2]{\bgroup\renewcommand*{\arraystretch}{1}\begin{bmatrix} #1 \\ #2\end{bmatrix} \egroup}

\newlength\cellsize 
\savebox2{%
\begin{picture}(12,12)
\put(0,0){\line(1,0){12}}
\put(0,0){\line(0,1){12}}
\put(12,0){\line(0,1){12}}
\put(0,12){\line(1,0){12}}
\end{picture}}

\newcommand\cellify[1]{\def\thearg{#1}\def\nothing{}%
\ifx\thearg\nothing
\vrule width0pt height\cellsize depth0pt\else
\hbox to 0pt{\usebox2\hss}\fi%
\vbox to 12\unitlength{
\vss
\hbox to 12\unitlength{\hss$#1$\hss}
\vss}}

\newcommand\tableau[1]{\vtop{\let\\=\cr
\setlength\baselineskip{-16000pt}
\setlength\lineskiplimit{16000pt}
\setlength\lineskip{0pt}
\halign{&\cellify{##}\cr#1\crcr}}}

\savebox3{%
\begin{picture}(12,12)
\put(0,0){\line(1,0){12}}
\put(0,0){\line(0,1){12}}
\put(12,0){\line(0,1){12}}
\put(0,12){\line(1,0){12}}
\end{picture}}

\newcommand\expath[1]{%
\hbox to 0pt{\usebox3\hss}%
\vbox to 12\unitlength{
\vss
\hbox to 12\unitlength{\hss$#1$\hss}
\vss}}

\newcommand\cell[3]{
\def\i{#1} \def\j{#2} \def\entry{#3}

\draw (\j-1,-\i)--(\j,-\i)--(\j,-\i+1)--(\j-1,-\i+1)--(\j-1,-\i);
\node at (\j-.5,-\i+.5) {\entry};
}

\newcommand{\qtrip}[3]{
\begin{tikzpicture}[scale=0.5]
\cell{1}{0}{#1} \cell{2}{0}{#2} \cell{2}{2.7}{#3}
\node at (1,-1.5) {$\cdots$};
\end{tikzpicture}
}

\begin{document}

\title[Macdonald polynomials and one-dimensional particle processes]{New formulas for Macdonald polynomials via the multispecies exclusion and zero range processes}


\author{Olya Mandelshtam}
\address{University of Waterloo, Waterloo, Ontario, Canada}
\curraddr{}
\email{omandels@uwaterloo.ca}
\thanks{The author was partially supported by NSERC grant RGPIN-2021- 02568 and NSF grant DMS-1953891.}


\date{May 24, 2024}

\begin{abstract}
We describe some recently discovered connections between one-dimensional interacting particle models and Macdonald polynomials. The first such model is the \emph{multispecies asymmetric simple exclusion process (ASEP)} on a ring, linked to the symmetric Macdonald polynomial $P_{\lambda}(X;q,t)$ through its partition function. Through this connection, a new formula was found for $P_{\lambda}$ by generalizing \emph{multiline queues}, which were introduced by Martin in 2018 to compute stationary probabilities of the ASEP. The second particle model is the \emph{multispecies totally asymmetric zero range process (TAZRP)} on a ring, which was very recently found to have an analogous connection to the modified Macdonald polynomial $\widetilde{H}_{\lambda}(X;q,t)$ through its partition function. This discovery coincided with a new formula for $\widetilde{H}_{\lambda}$, this time in terms of tableaux with a \emph{queue inversion} statistic, which also compute stationary probabilities of the TAZRP. We explain the plethystic relationship between multiline queues and queue inversion tableaux, and along the way, derive a new formula for $P_{\lambda}$ using the queue inversion statistic. This plethystic correspondence is closely related to fusion in the setting of integrable systems.
\end{abstract}

\maketitle

\section{Introduction}

The goal of this article is to describe some recent developments that explore the relationships between interacting particle models and Macdonald polynomials. Over the last couple of decades, a growing body of research has been devoted to studying these relationships, revealing remarkable connections between particle systems and symmetric functions. Much of this research has been focused on discovering new combinatorial objects and positive formulas with the goal of gaining a deeper understanding of the two areas. This article will focus on the role of two particle models which have been found to play a significant role in recent years: the \emph{asymmetric simple exclusion process} (ASEP) and the \emph{zero range process} (ZRP).

 The \emph{symmetric Macdonald polynomials} $P_{\lambda}(X;q,t)$, introduced by Macdonald \cite{Mac88}, are a multivariate family of orthogonal polynomials in the variables $X=\{x_1,x_2,\ldots\}$, characterized by certain triangularity and orthogonality axioms. They are indexed by partitions, and their coefficients are rational functions in the parameters $q$ and $t$. They notably generalize the Schur functions, Hall--Littlewood polynomials, Jack polynomials, and $q$-Whittaker functions. The \emph{modified Macdonald polynomials} $\HH_{\lambda}(X;q,t)$ are a transformed version of the $P_{\lambda}$'s, obtained through the formal operation of plethysm applied to a scalar multiple of the $P_{\lambda}$'s. The $\HH_{\lambda}$'s famously enjoy many ``nice'' properties such as being Schur positive, and appear in a variety of formulas with positive expansions in various bases. There has been much attention devoted to combinatorial study of the $P_{\lambda}$'s, the $\HH_{\lambda}$'s, and related functions, including the discovery of various combinatorial objects that lead to positive formulas for these polynomials.
 
 The initiation of the study of 1D non-equilibrium exactly solvable particle models is historically attributed to Spitzer in \cite{spitzer-1970}, where he introduced both the ASEP and the ZRP. In general, exclusion processes are a class of continuous time interacting particle models on a 1D lattice, such that there is at most one particle per site (hence, exclusion) -- with the ASEP as the most famous model in this class. On the other hand, zero range processes are in a sense defined to have the opposite property: any site can contain any number of particles. For both models, transitions are defined by particles hopping between adjacent sites. We consider \emph{multispecies} versions of both types of models, by assigning labels in $\mathbb{N}$ to the particles, so that particles with higher labels are considered stronger, and particles of the same label are indistinguishable. See \cref{sec:ASEP} and \cref{sec:TAZRP} for the complete definitions.
 
The ASEP enjoys rich combinatorial structure, which in many cases has been exploited to obtain combinatorial formulas for its exact stationary distributions. Another enticing aspect for algebraic combinatorics has been its intriguing connection to Macdonald polynomials. The first such link was found for the ASEP with open boundaries, whose partition function was discovered to be closely linked to Askey--Wilson polynomials in the single-species case \cite{sasamoto-1999,USW} and Koornwinder polynomials (Macdonald polynomials of type BC) in the multispecies case \cite{CW18,cantini-2017,cantini-etal-2016}. Remarkably, a similar connection was also discovered for classical (type A) Macdonald polynomials in \cite{CGW-2015}, who found that the partition function of the multispecies ASEP on a circular lattice is a specialization of $P_{\lambda}(X;q,t)$ at $x_1=\cdots=x_n=q=1$, with stationary probabilities appearing as certain nonsymmetric Macdonald polynomials. This connection was made explicit through two different approaches: via \emph{multiline queues}, which are probabilistic objects that project to the dynamics of the ASEP,  in \cite{CMW18}, and via vertex models, which are a related object coming from Yang--Baxter integrability,  in \cite{borodin-wheeler-2019}. The latter was generalized in several works (see \cite{ABW21} and the references therein) to colored vertex models to obtain a number of analogous formulas for various families of symmetric and nonsymmetric functions related to Macdonald polynomials, including the modified Macdonald polynomial $\HH_{\lambda}$ in \cite{garbali-wheeler-2020}. Independently, a new multiline-queue-esque formula for $\HH_{\lambda}$ was found in \cite{AMM20}, using a \emph{queue inversion} statistic on tableaux.  It was then a natural question to find a particle process analogous to the ASEP whose partition function corresponds to $\HH_{\lambda}$. In \cite{AMM22}, such a process was found to be the multispecies totally asymmetric zero range process on a circle (mTAZRP). 

The purpose of the present article is to tie together the ``$P_{\lambda}$ triangle'' and the ``$\HH_{\lambda}$ triangle'' of \cref{fig:triangles}. In \cref{sec:MLQ}, we describe the multiline queue formula that connects the multispecies ASEP on a circle to $P_{\lambda}$. In \cref{sec:mTAZRP tableaux} we describe the mTAZRP and a tableau formula that illustrates the connection to $\HH_{\lambda}$. In \cref{sec:plethysm} we give a tableaux characterization of the multiline queue formula and show how the the two formulas are tied together through plethysm. This analysis brings us to a new formula and a new conjectural compact formula for $P_{\lambda}(X;q,t)$, in terms of the queue inversion statistic. 

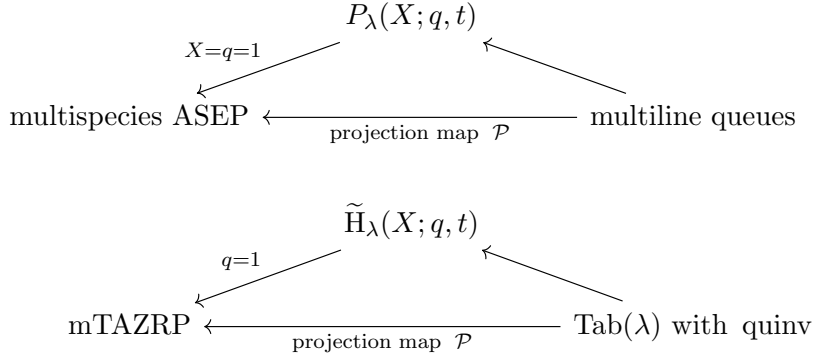
\begin{figure}[H]
\centering
\begin{tikzcd}
 & P_{\lambda}(X;q,t) \arrow[dl,swap,"X=q=1"]  \\
\text{multispecies ASEP} && \text{multiline queues} \arrow[ul] \arrow[ll,"\text{projection map}\ \proj"] \\
\vspace{0.2in}
 & \HH_{\lambda}(X;q,t) \arrow[dl,swap,"q=1"] \\
\text{mTAZRP} &&  \Tab(\lambda)\ \text{with}\ \quinv \arrow[ul] \arrow[ll,"\text{projection map}\ \proj"] 
\end{tikzcd}
\vspace{0.2in}
\caption{In the ``$P_{\lambda}$ triangle'', multiline queues are the combinatorial object interpolating between probabilities of the ASEP on a circle and the classical Macdonald polynomials. In the ``$\HH_{\lambda}$ triangle'', fillings of tableaux $\Tab(\lambda)$ with the \emph{queue inversion} statistic are the combinatorial object interpolating between the modified Macdonald polynomials and the multispecies TAZRP on a circle.}\label{fig:triangles}
\end{figure}

 \section{Notation and definitions}

\subsection{Combinatorial statistics and tableaux}\label{sec:stats}
A \emph{composition} is a sequence $\alpha=(\alpha_1,\dots, \alpha_k)$, of $k$ non negative integers, which are called its \emph{parts}. If some parts are 0, it is called a \emph{weak composition}. The number of nonzero parts is given by $\ell(\alpha)$, and $|\alpha|=\sum_i\alpha_i$ is the sum of the parts. Denote by $\lambda(\alpha)$ and $\inc(\alpha)$ the compositions obtained by rearranging the parts of $\alpha$ in weakly decreasing and increasing order, respectively.

If $\lambda$ is a weakly decreasing composition, we call it a partition, and we denote it as $\lambda\vdash n$ where $n:=|\lambda|$. We may equivalently use the \emph{frequency notation} to describe a partition $\lambda$ by giving the multiplicity of its parts. Let $m_i(\lambda)$ be the number of parts of size $i$ in $\lambda$; then we may write $\lambda$ as $\langle 1^{m_1(\lambda)}2^{m_2(\lambda)}\ldots\rangle$. If each part of $\lambda$ appears at most once, we call $\lambda$ a \emph{strict} partition.

Given a partition $\lambda=(\lambda_1,\dots, \lambda_n)$, its \emph{diagram} $\dg(\lambda)$ is a configuration of bottom justified columns with $\lambda_i$ cells in the $i$'th column from left to right. Make note that our definition of $\dg(\lambda)$ corresponds to the Ferrers graph of the conjugate partition $\lambda'$.
  
The columns of $\dg(\lambda)$ are labeled from left to right, and the rows from bottom to top. The notation $u=(r,j)$ refers to the cell in row $r$ and column $j$ (again, departing from the convention of Cartesian coordinates). We denote by $\South(u)=\South(r,j)$ the cell $(r-1,j)$ directly below $u$, if it exists; if $r=1$ and $\South(u)$ doesn't exist, we set the convention $\South(u)=\infty$. 

For example, $\lambda=(3, 3, 3, 2, 2, 1)$ may be written as $\langle 1^12^23^3\rangle$.  For $\alpha = (0,2,0,2,3,3,1, 3)$ then  $\inc(\alpha) = (0,0,1,2,2,3,3,3)$, $\lambda(\alpha) = \lambda$, and $\ell (\alpha) = 6$.  See \cref{fig:leg} for the diagram of $\lambda$.

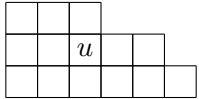
\begin{figure}[h]
  \centering
\begin{tikzpicture}[scale=.5]
\node at (0,0) {\tableau{\ &\ &\ \\\ &\ &u&\ &\ \\\ &\ &\ &\ &\ &\ }};

\end{tikzpicture} 
\caption{We show $\dg(\lambda)$ for $\lambda=\langle 1^1 2^2 3^3\rangle$. The cell $u=(2,3)$ has $\leg(u)=1$ and $\rarm(u)=3$ (from\cref{def:rarm}).}
\label{fig:leg}
\end{figure}

The \emph{leg} of a cell $(r,j)\in\dg(\lambda)$, denoted by $\leg(u)=\leg(r,j)$, is equal to $\lambda_j-r$: the number of cells above in the same column as $u$.

A \emph{filling} $\sigma:\dg(\lambda) \to \Z_+$ is an assignment of positive integers to the cells of $\dg(\lambda)$. For a cell $u\in\dg(\lambda)$, $\sigma(u)$ denotes the entry in the cell $u$. The content of a filling $\sigma$ is recorded as a monomial in the variables $X=x_1,x_2,\ldots$, and is denoted by $x^{\sigma} := \prod_{u\in \dg(\lambda)} x_{\sigma(u)}$. 

Define the \emph{set of descents} of a filling $\sigma$ of $\dg(\lambda)$ to be 
\[
	\Des(\sigma) = \{u\in \dg(\lambda)\ :\ \sigma(u)>\sigma(\South(u))\}.
\]   
The \emph{major index} $\maj(\sigma)$ is then defined as:
\[
\maj(\sigma) = \sum_{s \in \Des(\sigma)} (\leg(s)+1).
\]

We also define the notion of attacking cells and non-attacking fillings in our setting. (Our definition is the reverse of that of \cite{HHL05}, so as to be compatible with the queue inversion statistic of \cref{def:quinv}.) 
\begin{defn}\label{def:attacking}
Two cells $x=(r_1,i)$ and $y=(r_2,j)$ with $i<j$ are considered a \emph{quinv-attacking pair} if $r_1=r_2$ or $r_2=r_1-1$. That is, the cells are in one of the two configurations:
\begin{center}
\begin{tikzpicture}[scale=.5]
\node at (0,0) {\tableau{x&&&y}};
\node at (.2,0) {$\cdots$};
\node at (4,0) {or};
\node at (8,0) {\tableau{x\\&&&y}};
\node at (8,0) {$\ddots$};
\end{tikzpicture}
\end{center}
A filling $\sigma$ is called \emph{quinv-non-attacking} (or simply non-attacking if the context is clear) if $\sigma(x)\neq \sigma(y)$ for any quinv-attacking pair of cells $x$ and $y$.
\end{defn}

 \subsection{Multispecies ASEP}\label{sec:ASEP}
 The ASEP was originally introduced as a model for protein synthesis \cite{MGP68} and soon after was brought into the mathematical physics sphere by Spitzer in the late 1960s \cite{spitzer-1970} as a canonical example of an exactly solvable non-equilibrium process that exhibits boundary-induced phase transitions. Since the initiation of the field, the ASEP has been widely studied by physicists and mathematicians, and has since found connections and applications to many areas including random matrix theory, total positivity on the Grassmanian, formation of shocks, molecular transport, and others. 

The ASEP has a number of interesting variations, many of which are exactly solvable; the version we shall focus on in this article is a continuous time multispecies ASEP on a circular lattice with a fixed number of sites (i.e.~with periodic boundary conditions) and a hopping parameter $0 \leq t\leq 1$. In this model, particles are labeled by positive integers where the labels represent different "species"; our convention is that particles with higher labels are considered stronger. Moreover, for ease of notation, we shall consider vacancies as particles with the smallest label $0$. As the boundary is closed, particles are conserved, and so the state space is determined by a positive integer $n$ and a partition $\lambda=(\lambda_1,\ldots,\lambda_n)$, so that the $j$'th part of $\lambda$ corresponds to a particle of species $\lambda_j$, for each $j$. The state space of this process is enumerated by permutations of the parts of $\lambda$, and is denoted by $\mathcal{A}(\lambda,n)$. See \cref{fig:ASEP} for an example of a state of the ASEP on $n=7$ sites with particles of types $3,2,2,1$. 

The multispecies ASEP is a continuous time Markov chain on $\mathcal{A}(\lambda,n)$ with the following dynamics. For any pair of adjacent sites $j,j+1$ (considered cyclically modulo $n$, so that sites $n$ and $1$ are adjacent) their respective contents $a,b$ can swap to become $b,a$, such that the rate of swapping when $a>b$ is $t$ times the rate when $a<b$. (Particles with the same label are indistinguishable, so if $a=b$, no transition occurs.)

\begin{example}
For example, $\mathcal{A}((3,3,1),4)$ consists of the following 12 states: 
\begin{align*}(3,3,1,0), (3,1,3,0), (1,3,3,0), (3,3,0,1), (3,1,0,3), (1,3,0,3),\\
 (3,0,3,1), (3,0,1,3), (1,0,3,3), (0,3,3,1), (0,3,1,3), (0,1,3,3).\end{align*}
Examples of some transitions are:
\begin{itemize}
\item The jumps from $(3,1,3,0)$ are to $(1,3,3,0)$ and $(3,1,0,3)$ with rate $t$ and to $(0,1,3,3)$ and $(3,3,1,0)$ with rate $1$.
\item The jumps from $(0,3,3,1)$ are to $(1,3,3,0)$ with rate $t$ and to $(3,0,3,1)$ and $(0,3,1,3)$ with rate $1$.
\end{itemize}
\end{example}

\begin{remark} Although we describe the multispecies ASEP as a continuous time process, it it also natural to consider it in discrete time: at each time step, select a site with probability $1/n$, and then let the hopping rate $t$ dictate whether or not a hop occurs from that site. The discrete and continuous time processes will have the same stationary distribution.
\end{remark}

\begin{figure}[h!]
\centering
\begin{tikzpicture}[
  <-,   
  thick,
  main node/.style={circle, white,fill=none, draw},
]
  \newcommand*{\MainNum}{7}
  \newcommand*{\MainRadius}{1.5cm} 
  \newcommand*{\MainStartAngle}{90}

  \path
    (0, 0) coordinate (M)
    \foreach \t [count=\i] in {{\ \textcolor{black}{2}\ }, {\ \textcolor{black}{0}\ }, {\ \textcolor{black}{0}\ }, {\ \textcolor{black}{3}\ }, {\ \textcolor{black}{1}\ }, {\ \textcolor{black}{0}\ }, {\ \textcolor{black}{2}\ }} {
      +({\i-1)*360/\MainNum + \MainStartAngle}:\MainRadius)
      node[main node, align=center] (p\i) {\t}
    }
  ;  

  \foreach \i in {1, ..., \MainNum} {
    \pgfextracty{\dimen0 }{\pgfpointanchor{p\i}{north}} 
    \pgfextracty{\dimen2 }{\pgfpointanchor{p\i}{center}}
    \dimen0=\dimexpr\dimen2 - \dimen0\relax 
    \ifdim\dimen0<0pt \dimen0 = -\dimen0 \fi
    \pgfmathparse{2*asin(\the\dimen0/\MainRadius/2)}
    \global\expandafter\let\csname p\i-angle\endcsname\pgfmathresult
  }

  \foreach \i [evaluate=\i as \nexti using {int(mod(\i, \MainNum)+1}]
  in {1, ..., \MainNum} {  
    \pgfmathsetmacro\StartAngle{   
      (\i-1)*360/\MainNum + \MainStartAngle
      + \csname p\i-angle\endcsname
    }
    \pgfmathsetmacro\EndAngle{
      (\nexti-1)*360/\MainNum + \MainStartAngle
      - \csname p\nexti-angle\endcsname
    }
    \ifdim\EndAngle pt < \StartAngle pt
      \pgfmathsetmacro\EndAngle{\EndAngle + 360}
    \fi
    \draw
      (M) ++(\EndAngle:\MainRadius)
      arc[start angle=\EndAngle, end angle=\StartAngle, radius=\MainRadius]
    ;
  }
\end{tikzpicture}
\qquad\qquad\qquad
\begin{tikzpicture}[
  ->,   
  thick,
  main node/.style={circle, white,fill=none, draw},
]
  \newcommand*{\MainNum}{5}
  \newcommand*{\MainRadius}{1.5cm} 
  \newcommand*{\MainStartAngle}{90}

  \path
    (0, 0) coordinate (M)
    \foreach \t [count=\i] in {{\ \ \textcolor{black}{$\emptyset$} \ \ \ }, {\textcolor{black}{3,1,1}}, {\ \ \textcolor{black}{$\emptyset$}\ \ }, {\ \textcolor{black}{4,2,2}\ }, {\textcolor{black}{3,2,1}}} {
      +({\i-1)*360/\MainNum + \MainStartAngle}:\MainRadius)
      node[main node, align=center] (p\i) {\t}
    }
  ;  

   \foreach \i in {1, ..., \MainNum} {
    \pgfextracty{\dimen0 }{\pgfpointanchor{p\i}{north}} 
    \pgfextracty{\dimen2 }{\pgfpointanchor{p\i}{center}}
    \dimen0=\dimexpr\dimen2 - \dimen0\relax 
    \ifdim\dimen0<0pt \dimen0 = -\dimen0 \fi
    \pgfmathparse{2*asin(\the\dimen0/\MainRadius/2)}
    \global\expandafter\let\csname p\i-angle\endcsname\pgfmathresult
  }

  \foreach \i [evaluate=\i as \nexti using {int(mod(\i, \MainNum)+1}]
  in {1, ..., \MainNum} {  
    \pgfmathsetmacro\StartAngle{   
      (\i-1)*360/\MainNum + \MainStartAngle
      + \csname p\i-angle\endcsname
    }
    \pgfmathsetmacro\EndAngle{
      (\nexti-1)*360/\MainNum + \MainStartAngle
      - \csname p\nexti-angle\endcsname
    }
    \ifdim\EndAngle pt < \StartAngle pt
      \pgfmathsetmacro\EndAngle{\EndAngle + 360}
    \fi
    \draw
      (M) ++(\EndAngle:\MainRadius)
      arc[start angle=\EndAngle, end angle=\StartAngle, radius=\MainRadius]
    ;
  }
\end{tikzpicture}
\caption{On the left is a state of the ASEP of type $\lambda=\langle 1^1 2^2 3^1 \rangle$ on $n=7$ sites. Sites are labeled 1 through $n$ clockwise, starting from the topmost site. The given state is represented by the composition $(2,2,0,0,3,1,0)\in\ASEP(\lambda,n)$, with arrows showing the preferred hopping direction of the particles. On the right is a state of the TAZRP of type $\langle 1^3 2^3 3^2 4^1\rangle$ on 5 sites. The multiset composition representing this state is $(\,\cdot\,|\,321\,|\,422\,|\,\cdot\,|\,311\,)$, with arrows showing the clockwise direction of hopping in the TAZRP.}\label{fig:ASEP}
\end{figure}
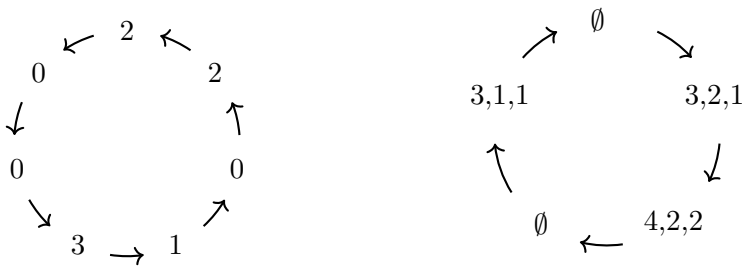

\subsection{Multispecies TAZRP}\label{sec:TAZRP}

The zero range process (ZRP) is a typical example of a 1D non-equilibrium process that exhibits condensation transitions. A ZRP can generally be defined on any graph (with possibly some additional data), where every site may contain an arbitrary number of particles. Particles hop from one site to the other in continuous time with a rate that depends solely on the content of the origin site, whence the name `zero range'. A key property of this process, and one of the many reasons the ZRP is of great interest to physicists, is that the stationary distribution is a product measure; see~\cite{EvansHanney} for a review. 

Like the ASEP, physicists and mathematicians have studied a range of different versions of the ZRP; the one we consider in this paper is called the \emph{multispecies totally asymmetric zero-range process} (mTAZRP), which had been introduced in \cite{takeyama-2015}, and is a specialization of a more general multispecies ZRP that is described in \cite{kuniba-okado-watanabe-2017}. 

For positive integer $n$ and a partition $\lambda=(\lambda_1,\ldots,\lambda_k)$, define $\T(\lambda,n)$ to be the set of arrangements of the $k$ particles labeled by the parts of $\lambda$ amongst the $n$ sites. $\T(\lambda,n)$ can be described as the set of multiset compositions of length $n$ that rearrange to $\lambda$: a state $w=(w_1,\ldots,w_n)\in\T(\lambda,n)$ is a sequence of (possibly empty) multisets such that the union of all the parts $\biguplus_{j=1}^n w_j$ is equal to the multiset $\lambda$.

The mTAZRP is a continuous-time Markov process on $\T(\lambda,n)$ with a global parameter $t\geq 0$ and site-dependent parameters $x_1,\ldots,x_n$. Each particle is equipped with an exponential clock, such that when a clock goes off, a particle jumps to an adjacent site (a particle at site $j$ jumps to site $j+1$ cyclically, so that a particle jumping out of site $n$ enters site $1$). Define the transition rates as follows. For each $j\in\{1,\dots,n\}$ and each species $r\geq 1$, if site $j$ contains $c_r$ particles of type $r$ and $d_r$ particles of type larger than $r$, then the total rate of jumps of particles of type $r$ from site $j$ is $x_j^{-1} t^{d_r}\sum_{i=0}^{c_r-1} t^{i}$. 

\begin{example}
$\T((3,1,1),3)$ consists of the following
$18$ states: 
\begin{align*}
(311|\cdot|\cdot), (31|1|\cdot), (31|\cdot|1), (3|11|\cdot), (3|1|1), (3|\cdot|11), (11|3|\cdot), (11|\cdot|3), (1|31|\cdot),\\ 
(1|1|3), (1|3|1), (1|\cdot|31), (\cdot|311|\cdot), (\cdot|31|1), (\cdot|3|11), (\cdot|11|3), (\cdot|1|31), (\cdot|\cdot|311).
\end{align*}
Examples of transitions of the mTAZRP on $\T((3,1,1),3)$ are: 
\begin{itemize}[itemsep=0pt]
\item The jumps from $(\cdot|311|\cdot)$ are to $(\cdot|11|3)$ with rate $x_2^{-1}$, and
to $(\cdot|31|1)$ with rate $x_2^{-1}(t+t^2)$;
\item The jumps from $(\cdot|1|31)$ are to $(3| 1| 1)$ with rate $x_3^{-1}$, 
to $(1| 1|3)$ with rate $x_3^{-1}t$, and to $(\cdot|\cdot|311)$ with rate $x_2^{-1}$. 
\end{itemize}
\end{example}

\section{The ASEP and multiline queues}\label{sec:MLQ}

The concept of using multiline queues as a tool to study the ASEP was famously introduced by Ferrari and Martin in \cite{FM07} to compute probabilities for the \emph{totally} asymmetric simple exclusion process (TASEP), which is the $t=0$ case. More recently, Martin generalized the result for the general $t$ setting in \cite{martin-2020}. Corteel, Williams, and the author introduced a modification of Martin's multiline queues in \cite{CMW18}, also computing probabilities for the ASEP. In this section, we focus on the multiline queues defined in \cite{CMW18}. 

\begin{defn}[Ball system]\label{def:MLQ} Fix a positive integer $n$ and a partition $\lambda$ with $\ell(\lambda)\leq n$. A \emph{ball system} $B$ of size $\lambda,n$ is an $L \times n$, where $L:=\lambda_1$, with rows labeled 1 through $L$ from bottom to top and columns labeled $1$ through $n$ from left to right. The array is a cylinder, so that column $n$ is adjacent to column 1. Each of the $Ln$ positions in the array is either empty or occupied by a ball, such that row $j$ contains exactly $\lambda_j'$ balls. We can represent a ball system by an $L$-tuple of subsets of $n$ where the $j$'th subset corresponds to the set of columns that contain balls in row $j$. Denote the set of ball systems of size $\lambda,n$ by $\B(\lambda,n)$. Then
\[
\B(\lambda,n)=\{(b_1,\ldots,b_L):  b_j\subseteq \{1,\ldots,n\}\ \mbox{and}\ |b_j|=\lambda_j'\ \mbox{for}\ 1 \leq j \leq L\}.
\]
\end{defn}
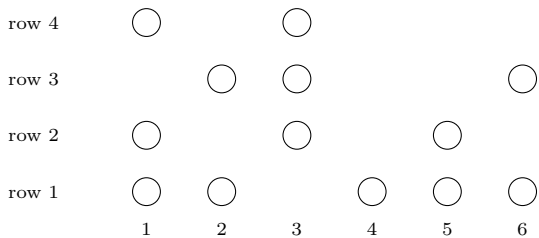
\begin{figure}[!ht]
\centering
\begin{tikzpicture}[scale=0.5]
\def \h {1.5}
\def \w {2}
\def \r {.37cm}

\draw (1*\w,3*\h) circle (\r) node[color=black] {$\ $};
\draw (3*\w,3*\h) circle (\r) node[color=black] {$\ $};

\draw (2*\w,2*\h) circle (\r) node[color=black] {$\ $};
\draw (3*\w, 2*\h) circle (\r) node[color=black] {$\ $};
\draw (6*\w,2*\h) circle (\r) node[color=black] {$\ $};

\draw (1*\w,\h) circle (\r) node[color=black] {$\ $};
\draw (3*\w,\h) circle (\r) node[color=black] {$\ $};
\draw (5*\w,\h) circle (\r) node[color=black] {$\ $};

\draw (1*\w,0) circle (\r) node[color=black] {$\ $};
\draw (2*\w,0) circle (\r) node[color=black] {$\ $};
\draw (4*\w,0) circle (\r) node[color=black] {$\ $};
\draw (5*\w,0) circle (\r) node[color=black] {$\ $};
\draw (6*\w,0) circle (\r) node[color=black] {$\ $};

\node at (-1,3*\h) {\tiny row $4$};
\node at (-1,2*\h) {\tiny row $3$};
\node at (-1,\h) {\tiny row $2$};
\node at (-1,0) {\tiny row $1$};
\node at (\w,-1) {\tiny $1$};
\node at (2*\w,-1) {\tiny $2$};
\node at (3*\w,-1) {\tiny $3$};
\node at (4*\w,-1) {\tiny $4$};
\node at (5*\w,-1) {\tiny $5$};
\node at (6*\w,-1) {\tiny $6$};

\end{tikzpicture}
 \caption{
A ball system of type $\lambda=(4,4,3,1,1)$ and $n=6$, represented as the tuple $B=(\{1,2,4,5,6\},\{1,3,5\},\{2,3,6\},\{1,2\})$.}
\label{parameters2}
 \end{figure}
 
\begin{defn}[Pairing]
Let $B\in\B(\lambda,n)$ be a ball system. A \emph{pairing} of $B$ is an algorithm by which balls are labeled and connected by strands from higher rows to lower rows according to the following rules.
\begin{itemize}
\item Every ball in a row $r>1$ of $B$ is matched to a unique ball in row $r-1$. Two balls are connected by the shortest strand (looking weakly to the right and wrapping if necessary around the cylinder) to connect the two balls. We abuse the definition slightly by also calling every such strand a ``pairing''.
    \item The label of each ball is equal to the length of the strand containing it. (In other words, a strand originating at row $r$ has all balls contained in it labeled "$r$".) 
    \item If a ball with label $a$ is directly above a ball with label $a$ in the same column, then necessarily, the two balls are connected by a strand. Such pairings are called \emph{trivial pairings}.
    \item If a ball with label $a$ is directly above a ball with label $b$ in the same column, then $a<b$.
\end{itemize}
If a strand is pairing balls from row $r$ to row $r-1$, the ball in row $r$ is called the \emph{departure site} and the ball in row $r-1$ is called the \emph{arrival site}.
\end{defn}

\begin{remark}
The above definition of a pairing, corresponding to the multiline queues defined in \cite{CMW18}, is a variation on the multiline queues introduced by Martin in \cite{martin-2020}: in Martin's definition, the third bullet point is missing. Our definition results in a slightly different (and generally smaller) set of multiline queues corresponding to any given size. We revisit this difference in the discussion of \cref{sec:comparison}.
\end{remark}

See \cref{parameters2} for an example of a ball system of size $(4,4,3,1,1)$ and $n=6$, and \cref{fig:MLQexample} for an example of a pairing on that ball system.

Note that a pairing uniquely determines the labeling, but not vice versa-- it is possible to have two different pairings corresponding to a single labeling of a ball system, for example:
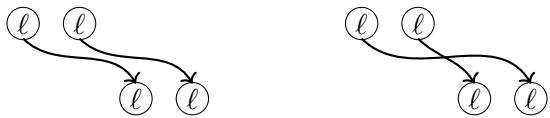
\[
\begin{tikzpicture}[scale=0.5]
\def \h {2}
\def \w {1.5}
\def \r {.43cm}
\def \e {9}

\draw (0*\w,1*\h) circle (\r) node[color=black] {$\ell$};
\draw (1*\w,1*\h) circle (\r) node[color=black] {$\ell$};
\draw[thick,->,black] (0*\w, 1*\h-.4) to[out=-45,in=120] (2*\w, 0*\h+.4);
\draw[thick,->,black] (1*\w, 1*\h-.4) to[out=-45,in=120] (3*\w, 0*\h+.4);
\draw (2*\w,0*\h) circle (\r) node[color=black] {$\ell$};
\draw (3*\w,0*\h) circle (\r) node[color=black] {$\ell$};

\draw (\e+0*\w,1*\h) circle (\r) node[color=black] {$\ell$};
\draw (\e+1*\w,1*\h) circle (\r) node[color=black] {$\ell$};
\draw[thick,->,black] (\e+0*\w, 1*\h-.4) to[out=-45,in=120] (\e+3*\w, 0*\h+.4);
\draw[thick,->,black] (\e+1*\w, 1*\h-.4) to[out=-45,in=120] (\e+2*\w, 0*\h+.4);
\draw (\e+2*\w,0*\h) circle (\r) node[color=black] {$\ell$};
\draw (\e+3*\w,0*\h) circle (\r) node[color=black] {$\ell$};
 
\end{tikzpicture}
\]

\begin{defn}[Multiline queue]\label{def:MLQ}
A \emph{multiline queue} of size $\lambda,n$ is a pairing on a ball system in $\B(\lambda,n)$, and the set of them is denoted by $\MLQ(\lambda,n)$. The \emph{state} of a multiline queue is the sequence of integers $\proj(M)=(\alpha_1,\ldots,\alpha_n)$ read off the labels of the balls in the bottom row (from left to right), with the vacancies read as 0's. The state of a multiline queue can be thought of as the \emph{projection} of the multiline diagram onto its bottom row, and we denote this projection map by $\proj:\MLQ(\lambda,n)\rightarrow \ASEP(\lambda,n)$. Denote the set of multiline queues corresponding to state $\alpha=(\alpha_1,\ldots,\alpha_n)$ by $\MLQ(\alpha)$. (It is immediate from the definitions that if $M\in\MLQ(\lambda,n)$ and $\alpha=\proj(M)$, then $\lambda(\alpha)=\lambda$.)
\end{defn}

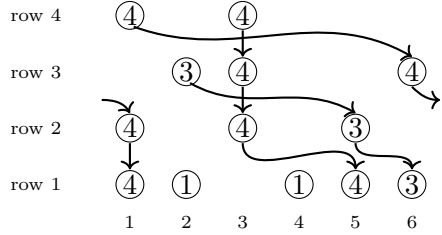
\begin{figure}[!ht]
\centering

\begin{tikzpicture}[scale=0.5]
\def \h {1.5}
\def \w {1.5}
\def \r {.37cm}

\draw (1*\w,3*\h) circle (\r) node[color=black] {$4$};
\draw (3*\w,3*\h) circle (\r) node[color=black] {$4$};

\draw[thick,->,black] (1*\w+.1, 3*\h-.4+.1) to[out=-20,in=150] (6*\w, 2*\h+.4);
\draw[thick,->,black] (3*\w, 3*\h-.4) to[out=-90,in=90] (3*\w, 2*\h+.4);

\draw (2*\w,2*\h) circle (\r) node[color=black] {$3$};
\draw (3*\w, 2*\h) circle (\r) node[color=black] {$4$};
\draw (6*\w,2*\h) circle (\r) node[color=black] {$4$};

\draw[thick,->,black] (2*\w+.1, 2*\h-.4+.1) to[out=-30,in=150] (5*\w, 1*\h+.4);
\draw[thick,->,black] (3*\w, 2*\h-.4) to[out=-90,in=90] (3*\w, 1*\h+.4);
\draw[thick,->,black] (6*\w, 2*\h-.4) to[out=-45,in=180] (6.5*\w, 1.5*\h);
\draw[thick,->,black] (0.5*\w, 1.5*\h) to[out=0,in=120] (1*\w, 1*\h+.4);

\draw (1*\w,\h) circle (\r) node[color=black] {$4$};
\draw (3*\w,\h) circle (\r) node[color=black] {$4$};
\draw (5*\w,\h) circle (\r) node[color=black] {$3$};

\draw[thick,->,black] (1*\w, 1*\h-.4) to[out=-90,in=90] (1*\w, 0*\h+.4);
\draw[thick,->,black] (3*\w, 1*\h-.4) to[out=-90,in=90] (5*\w, 0*\h+.4);
\draw[thick,->,black] (5*\w, 1*\h-.4) to[out=-90,in=90] (6*\w, 0*\h+.4);

\draw (1*\w,0) circle (\r) node[color=black] {$4$};
\draw (2*\w,0) circle (\r) node[color=black] {$1$};
\draw (4*\w,0) circle (\r) node[color=black] {$1$};
\draw (5*\w,0) circle (\r) node[color=black] {$4$};
\draw (6*\w,0) circle (\r) node[color=black] {$3$};

\node at (-1,3*\h) {\tiny row $4$};
\node at (-1,2*\h) {\tiny row $3$};
\node at (-1,\h) {\tiny row $2$};
\node at (-1,0) {\tiny row $1$};
\node at (\w,-1) {\tiny $1$};
\node at (2*\w,-1) {\tiny $2$};
\node at (3*\w,-1) {\tiny $3$};
\node at (4*\w,-1) {\tiny $4$};
\node at (5*\w,-1) {\tiny $5$};
\node at (6*\w,-1) {\tiny $6$};

\end{tikzpicture}
\caption{A multiline queue $M$ of size $\lambda=(4,4,3,1,1)$ with $n=6$, projecting to the state $\proj(M)=(4,1,0,1,4,3)\in\ASEP(\lambda,n)$.}
\label{fig:MLQexample}
 \end{figure}

We now define a weight $\wt(M)$ associated to each multiline queue $M$, as a function of the variables $X$ and parameters $q,t$. For each $M$, $\wt(M)$ is a monomial in $X$ multiplied by a rational function in $q,t$.

\begin{defn}\label{def:pairing}
Let $M$ be a multiline queue of size $(\lambda,n)$ corresponding to the ball system $B=(b_1,\ldots,b_L)$ where $L:=\lambda_1$. Define $x^M$ to be the content monomial $\prod_{r=1}^L \prod_{j\in b_r} x_j$. In other words, the power of $x_j$ is equal to the number of balls in column $j$ of $M$. 
\end{defn}

For example, the multiline queue in \cref{fig:MLQexample} has content $x^M=x_1^3x_2^2x_3^3x_4x_5^2x_6^2$.

We also define the \emph{$q,t$-weight} of $M$ by associating a weight to each pairing strand $p$. We shall consider the strands being created one at a time, pairing from the departure site to the arrival site. The pairing order must satisfy the following: 
\begin{itemize}
 \item strands are ordered by rows from top to bottom;
 \item within each row, strands are ordered from highest to lowest label; 
 \item within each row, for any given label, the trivial pairings must come first.
 \end{itemize}
Let $\pi$ be a fixed pairing order. For each ball $u$ in rows $2,\ldots,L$, define $d_{\pi}(u)$ to be the position in $\pi$ of the pairing with $u$ as its departure site. For each ball $u$ in rows $1,\ldots,L-1$ with label $\geq 2$, define $a_{\pi}(u)$ to be the position in $\pi$ of the pairing with $u$ as its arrival site. Call $\dep_{\pi}(M)$ the list of values of $d_{\pi}(u)$ ranging over all balls $u$ in rows $2,\ldots,L$, read row by row from top to bottom and from left to right within each row. For example, in \cref{fig:MLQexample}, there is only one possible ordering, given by $\dep_{\pi}(M)=(2,1,5,3,4,6,7,8)$.
 
We require that if the labelings of two multiline queues $M_1$ and $M_2$ coincide with respective pairing orders $\pi_1$ and $\pi_2$, then the orders of their departures are equal: $\dep_{\pi_1}(M_1)=\dep_{\pi_2}(M_2)$.

\begin{defn}\label{def:MLQ weight}
Let $M=(b_1,\ldots,b_L)$ be a multiline queue with pairing order $\pi$; $\wt_{\pi}(M)$ depends on this order. Let $p$ be a pairing of balls $u=(r,i)$ from row $r$ to $v=(r-1,j)$ with the label $\ell$. Define $\leg(p):=\ell-r$. If $p$ is a trivial pairing, it has weight 1. Otherwise, 
\[
\wt_{\pi}(p)(q,t)=\frac{t^{\Skip(p)}q^{\xi(p)(\leg(p)+1)}(1-t)}{1-t^{\Free(p)}q^{\leg(p)+1}},
\]
where $\xi(p)=1$ if $p$ wraps around the cylinder (i.e.~$i>j$) and $0$ otherwise. $\Skip(p)$ is the number of balls $x$ in row $r-1$ \emph{passed over} by $p$, whose arrival time $a_{\pi}(x)$ comes after the departure time $d_{\pi}(u)$ of $p$. $\Free(p)$ is the total number of balls in row $r-1$ whose arrival time comes after $d_{\pi}(u)-1$. Formally, 
\begin{align*} \Free(p)&=|\{x\in b_{r-1}:d_{\pi}(u)\leq a_{\pi}(x)\}|,\quad\text{and}\\
\Skip(p)&=|\{x=(r-1,k)\in b_{r-1}:d_{\pi}(u)<a_{\pi}(x)\ \text{and}\ i<k<j\pmod n\}|,\end{align*}
Now we define $\wt_{\pi}(M)=\prod_p \wt_{\pi}(p)$ to be the product of the weights over all the pairings in $M$. See \cref{ex:weight} for a computation of the weight of a multiline queue.
\end{defn}

Somewhat surprisingly, the total weight over all possible multiline queues with the same labeling is independent of the chosen pairing order! This is due to the following lemma.

\begin{lemma}[{\cite[Lemma 2.1]{CMW18}}]
Consider a ball system with a labeling that is compatible with some pairing. Let $\pi_1$ and $\pi_2$ be two fixed orderings corresponding to that labeling. Let $\mathcal{O}$ be the set of possible pairings corresponding to that labeling. Then,
\[
\sum_{M\in \mathcal{O}}\wt_{\pi_1}(M)=\sum_{M\in \mathcal{O}}\wt_{\pi_2}(M).
\]
That is, the total weight of all possible pairings for a fixed labeling is independent of the order in which we pair those balls. 
\end{lemma}

This allows us to fix a canonical ordering $\pi$ (say, ordering balls from left to right when there is a choice), and drop the subscript when we refer to the weight $\wt(M)$ of a multiline queue. Thus we can present the definition of $\wt(M)$ as follows. Call $\Skip(M)=\sum_{p}\Skip(p)$ where the sum is over all non-trivial pairings $p$ (according to a fixed pairing order), and $\maj(M)=\sum_{p:\xi(p)=1} \leg(p)+1$. Then we may write
\begin{equation}\label{eq:wtM}
\wt(M)=q^{\maj(M)} t^{\Skip(M)} \prod_{p \nontriv}\frac{1-t}{1-q^{\leg(p)+1}t^{\Free(p)}}.
\end{equation}

\begin{remark}
Our definition of $\wt(M)$ is based on \cite{CMW18}, and is again a generalization with the added parameters $X, q$ of a modification of the weight function on multiline queues in \cite{martin-2020}. Specifically, the difference in the definitions comes from our rule for trivial pairings.
\end{remark}

For a (weak) composition $\alpha$ and $n=\ell(\alpha)$, define
\begin{equation}\label{eq:F}
F_{\alpha}(x_1,\ldots,x_n; q, t)  = \sum_{M\in\MLQ(\alpha)} x^M \wt(M)
 \end{equation}
 and for a partition $\lambda$ with $\ell(\lambda)\leq n$, define
\begin{align} 
	Z_{\lambda}(X; q, t) =Z_{\lambda}(x_1,\dots,x_n; q, t)&= \sum_{M\in\MLQ(\lambda,n)} x^M\wt(M)\label{eq:Z}\\ 
	&= \sum_{\alpha:\lambda(\alpha)=\lambda} F_{\alpha}(X; q, t).\nonumber
\end{align}

Notably, it is shown in \cite{CMW18} that $F_{\alpha}(X;q,t)$ is precisely the \emph{permuted basement Macdonald polynomial} $E_{\inc(\alpha)}^{\nu}(X;q,t)$, where $\nu$ is the longest permutation such that $\nu\circ\alpha=\inc(\alpha)$ and $\inc(\alpha)$ is the composition obtained by rearranging $\alpha$ to be weakly increasing. Combining with the result of \cite{CGW-2015}, it follows that $F_{\alpha}(X;q,t)$ interpolates between the permuted basement Macdonald polynomials and the stationary probabilities of the multispecies ASEP.

\begin{theorem}[{\cite[Theorem 1.11]{CMW18}}]\label{thm:main mlq}
Let $\lambda$ be a partition and $n$ a positive integer. Let $\alpha$ be a composition such that $\lambda(\alpha)=\lambda$. The stationary probability of the state $\alpha\in\ASEP(\lambda,n)$ is given by
\[
\Pr(\alpha) = \frac{1}{\mathcal{Z}_{\lambda}(t)}F_{\alpha}(x_1=\cdots=x_n=1;1,t), 
\]
where 
\[
\mathcal{Z}_{\lambda}(t)=Z_{\lambda}(x_1=\cdots=x_n=1;1,t). 
\]
Moreover, the symmetric Macdonald polynomial is equal to $Z_{\lambda}(X;q,t)$, which can be written as
\begin{equation}\label{eq:P mlq}
P_{\lambda}(X;q,t)= \sum_{M\in\MLQ(\lambda,n)} x^M q^{\maj(M)} t^{\Skip(M)} \prod_{p \nontriv}\frac{1-t}{1-q^{\leg(p)+1}t^{\Free(p)}}.
 \end{equation}
\end{theorem}

\begin{example}\label{ex:weight}
We compute $\wt(M)$ for the multiline queue $M$ in \cref{fig:MLQexample}. Fix a pairing order $(p_1,\ldots,p_8)$ and let $\dep(M)=(2,1,5,3,4,6,7,8)$ be the corresponding order of departures when scanning the rows from top to bottom and from left to right within each row. Thus $p_1$ is the pairing from $(4,3)$ to $(3,3)$ which is $\dep(M)_2$, $p_2$ is the pairing from $(4,1)$ to $(3,6)$ which is $\dep(M)_1$, etc.
\begin{itemize}
\item $p_1, p_3,p_6$ are trivial, so $\wt(p_1)=\wt(p_3)=\wt(p_6)=1$.
\item $p_2$ skips over $(3,2)$, $\leg(p_2)=0$, and $\Free(p_2)=2$, so $\wt(p_2)=\frac{t(1-t)}{1-qt^2}$.
\item $p_4$ wraps around the cylinder so $\xi(p_4)=1,\leg(p_4)=2,\Skip(p_4)=0,\Free(p_4)=2$, so $\wt(p_4)=\frac{q^2(1-t)}{1-q^2t^2}$.
\item Finally, $\wt(p_5)=\frac{1-t}{1-qt}$, $\wt(p_7)=\frac{t(1-t)}{1-q^3t^4}$, and $\wt(p_5)=\frac{1-t}{1-q^2t^3}$.
\end{itemize}
Thus 
\[\wt(M)=\frac{q^2t^2(1-t)^5}{(1-q^3t^4)(1-q^2t^3)(1-q^2t^2)(1-qt^2)(1-qt)}.\]
\end{example}

\subsection{The integral form $J_{\lambda}$}\label{sec:J}
As described in \cite[\S 4]{CHMMW20}, we get for free a multiline queue formula for $J_{\lambda}(X;q,t)$, the \emph{integral form} of the Macdonald polynomial $P_{\lambda}$. The integral form is defined as $J_{\lambda}(X;q,t)=\Pi_{\lambda}(q,t)P_{\lambda}(X;q,t)$, where
\[\Pi(\lambda) = \prod_{x\in M} (1-q^{\leg(x)}t^{\Free(x)+1}).\]
Looking at the contribution to $\Pi(\lambda)$ from cells $x$ for which $\leg(x)=0$, we get a factor of $\prod_{i\geq 1} (t;t)_{m_i}$, where $(t;t)_k =(1-t)(1-t^2)\cdots(1-t^k)$ and $\lambda=\langle 1^{m_1}2^{m_2}\cdots\rangle$ (i.e.~for $i\geq 1$, $m_i$ is the number of parts of length $i$ in $\lambda$). We shall write this term as $(1-t)^{\ell(\lambda)}\perm(\lambda)$, as it is precisely the $t$-analog of the number of ways of permuting the columns of the same height in $\dg(\lambda)$. We will revisit the statistic $\perm$ in \cref{def:perm}.

The multiline queue formula for $P_{\lambda}$ can be modified to describe $J_{\lambda}$ as follows. From \cref{def:MLQ weight}, we observe that if every pairing of a multiline queue of size $\lambda,n$ is non-trivial, then the product of the denominators $(1-q^{\ell-r+1}t^{\Free(p)})$ over all pairings $p$ precisely equals $\Pi(\lambda)(1-t)^{-\ell(\lambda)}\perm(\lambda)^{-1}$. Thus we only need to consider those pairings which do not contribute a factor in the denominator to the product, namely the trivial pairings. Then,
\begin{align}
J_{\lambda}(X;q,t)& = (1-t)^{\ell(\lambda)}\perm(\lambda)\sum_{M\in\MLQ(\lambda)} x^M q^{\maj(M)} t^{\Skip(M)} \times \prod_{p\ \nontriv}(1-t) \nonumber\\
&\hspace{2in}\times\prod_{p\ \triv}(1-q^{\leg(p)+1}t^{\Free(p)}) \label{eq:J}.
\end{align}

\subsection{Multiline queues as tableaux}\label{sec:tableaux}

There is a natural mapping from a multiline queue of size $(\lambda,n)$ to a filling of $\dg(\lambda)$ with integers $\{1,\ldots,n\}$. We correspond each strand of length $\ell$ to a column of length $\ell$ in $\dg(\lambda)$ by filling the $j$'th box of that column (from bottom to top) with the site number of the ball in the strand at row $j$. We name this queue-tableau map $\mathcal{q}$ (at this stage, it is only well-defined when $\lambda$ is a strict partition). We then impose an ordering on columns of the same height by requiring that the entries at row 1 are in increasing order: this makes $\mathcal{q}$ a bijection. See \cref{fig:queue-tableau-map} for an example. 

\begin{figure}[h!]
\centering
\begin{tikzpicture}[scale=0.5]
\def \h {1.2}
\def \w {1.2}
\def \r {.37cm}

\draw (1*\w,3*\h) circle (\r) node[color=black] {$4$};
\draw (3*\w,3*\h) circle (\r) node[color=black] {$4$};

\draw[->,black] (1*\w+.1, 3*\h-.4+.1) to[out=-20,in=150] (6*\w, 2*\h+.4);
\draw[->,black] (3*\w, 3*\h-.4) to[out=-90,in=90] (3*\w, 2*\h+.4);

\draw (2*\w,2*\h) circle (\r) node[color=black] {$3$};
\draw (3*\w, 2*\h) circle (\r) node[color=black] {$4$};
\draw (6*\w,2*\h) circle (\r) node[color=black] {$4$};

\draw[->,black] (2*\w+.1, 2*\h-.4+.1) to[out=-30,in=150] (5*\w, 1*\h+.4);
\draw[->,black] (3*\w, 2*\h-.4) to[out=-90,in=90] (3*\w, 1*\h+.4);
\draw[->,black] (6*\w, 2*\h-.4) to[out=-45,in=180] (6.5*\w, 1.5*\h);
\draw[->,black] (0.5*\w, 1.5*\h) to[out=0,in=120] (1*\w, 1*\h+.4);

\draw (1*\w,\h) circle (\r) node[color=black] {$4$};
\draw (3*\w,\h) circle (\r) node[color=black] {$4$};
\draw (5*\w,\h) circle (\r) node[color=black] {$3$};

\draw[->,black] (1*\w, 1*\h-.4) to[out=-90,in=90] (1*\w, 0*\h+.4);
\draw[->,black] (3*\w, 1*\h-.4) to[out=-90,in=90] (5*\w, 0*\h+.4);
\draw[->,black] (5*\w, 1*\h-.4) to[out=-90,in=90] (6*\w, 0*\h+.4);

\draw (1*\w,0) circle (\r) node[color=black] {$4$};
\draw (2*\w,0) circle (\r) node[color=black] {$1$};
\draw (4*\w,0) circle (\r) node[color=black] {$1$};
\draw (5*\w,0) circle (\r) node[color=black] {$4$};
\draw (6*\w,0) circle (\r) node[color=black] {$3$};

\node at (-1,3*\h) {\tiny row $4$};
\node at (-1,2*\h) {\tiny row $3$};
\node at (-1,\h) {\tiny row $2$};
\node at (-1,0) {\tiny row $1$};
\node at (\w,-1) {\tiny $1$};
\node at (2*\w,-1) {\tiny $2$};
\node at (3*\w,-1) {\tiny $3$};
\node at (4*\w,-1) {\tiny $4$};
\node at (5*\w,-1) {\tiny $5$};
\node at (6*\w,-1) {\tiny $6$};

\draw[thick,->,scale=1.3] (6.5,1)--(9,1) node[midway,above] {$\mathcal{q}$};

\node[scale=1.3] at (15,1.6) {\tableau{1&3\\6&3&2\\1&3&5\\1&5&6&2&4}};

\end{tikzpicture}
\caption{The map $\mathcal{q}$ from a multiline queue to a tableau of the same size.}\label{fig:queue-tableau-map}
\end{figure}
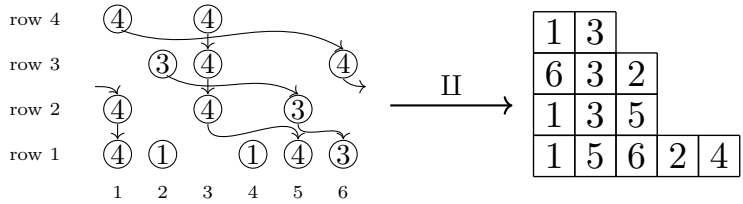

We observe a few things about how MLQ rules and statistics translate to rules and statistics on tableaux via this map. Let $M\in\MLQ(\lambda,n)$.
\begin{itemize}
\item[i.] At each row $r$ of the MLQ, there is at most one ball per site, so each entry in the set $\{1,\ldots,n\}$ can appear at most once in each row of the tableau.
\item[ii.] If a ball of label $i$ is above a ball of label $j$ in the MLQ, then either $i<j$, or $i=j$ and they are in the same strand. This means that if the same entry appears in rows $r$ and $r-1$ of the tableau, i.e. $\sigma(r,k)=\sigma(r-1,m)$, then $k>m$ and $\lambda_k<\lambda_m$ (the entry below is to the southwest in a larger column), or they are in the same column ($k=m$). In particular, they can never be an \emph{attacking pair} (see \cref{def:attacking}). 
\item[iii.] If a strand $p$ from a ball $(r,i)$ to a ball $(r-1,j)$ is wrapping, then $i>j$, corresponding to a \emph{descent} in the filling coming from the pair of cells $\tableau{i\\j}$, and $\leg(p)=\leg(u)$ where $u$ is the cell in the tableau containing the $i$. Thus $\maj(M)=\maj(\mathcal{q}(M))$.
\item[iv.] $\Free(p)$ and $\Skip(p)$ correspond to certain ``arm'' and ``coinversion'' statistics, which have a technical definition that can be found in \cite[\S 5]{CMW18}. However, in \cref{conj:compact} we will see a simplifying modification. 
\item[v.] We can naturally adapt the projection map $\proj$ for fillings by passing through $\mathcal{q}$: reading off the bottom row of a filling $\sigma$, for each $j$, the cell $(1,j)$ corresponds to a particle of type $\lambda_j$ at site $\sigma(1,j)$ of the ASEP of size $\lambda,n$.
\end{itemize}

We observe that these tableaux are suspiciously similar to \emph{permuted basement tableaux}, introduced by Ferreira in \cite{Fer11} and further studied by Alexandersson in \cite{Ale16}. To be precise, if we reflect the tableaux across the $y$-axis, we \emph{almost} obtain the permuted basement tableaux corresponding to $E_{\inc(\lambda)}^{\nu}$ for some permutation $\nu$. This is, of course, to be expected, since $F_{\alpha}(X;q,t) = E_{\inc(\alpha)}^{\nu}$ where $\nu\circ\alpha=\inc(\alpha)$ from \eqref{eq:F}. See \cite[Remark 5.7]{CMW18} for details.

\section{The modified Macdonald polynomials $\HH_{\lambda}(X;q,t)$}\label{sec:mTAZRP tableaux}

The modified Macdonald polynomials $\HH_{\lambda}(X;q,t)$ were originally introduced by Garsia--Haiman in \cite{GarsiaHaiman96}, and defined using plethystic notation in terms of the $J_{\lambda}$'s (the integral form of $P_{\lambda}$ defined in \cref{sec:J}) as
\begin{equation}\label{eq:H}
\HH_{\lambda}(X;q,t)=t^{n(\lambda)}J_{\lambda}\Big[\frac{X}{1-t^{-1}};q,t^{-1}\Big]
\end{equation}
where $n(\lambda)=\sum_i {\lambda_i'\choose 2}$. 

We postpone delving into the details of this formula until \cref{sec:plethysm}, and instead turn our focus to the combinatorics of the monomial expansion of $\HH_{\lambda}$.

\subsection{Queue inversion tableaux formula}

Define $\Tab(\lambda,n)$ to be the set of fillings of the diagram $\dg(\lambda)$ with the integers $1,\ldots,n$. Recall the celebrated Haglund--Haiman--Loehr formula for the modified Macdonald polynomials:
\begin{equation}\label{eq:HHL}
\HH_{\lambda}(x_1,\ldots,x_n;q,t) = \sum_{\sigma\in\Tab(\lambda,n)} x^{\sigma}q^{\maj(\sigma)}t^{\inv(\sigma)}.
\end{equation}
Here $\maj$ is from \cref{sec:stats}, and $\inv$ is a classical statistic introduced by Haglund--Haiman--Loehr in \cite{HHL05}.

In \cite{AMM20}, an analogous formula was given for $\HH_{\lambda}$ over the same set of fillings, but with the \emph{queue inversion} statistic ``$\quinv$'' instead of ``$\inv$'', the motivation for which is discussed in \cref{sec:plethysm}. We describe this statistic. For comparison to the inv statistic, please refer to \cite{HHL05}. 
\begin{defn}[Queue inversions]\label{def:quinv}
Given a diagram $\dg(\lambda)$, an \emph{$L$-triple} consists of either
\begin{itemize}
\item three cells $x=(r+1,i)$, $y=(r,i)$ and $z=(r,j)$ with $i<j$; or
\item two cells $y=(r,i)$ and $z=(r,j)$ with $i<j$ and  $\lambda_i = r$
(in which case the triple is called a \emph{degenerate triple}).
\end{itemize}
We shall refer to $L$-triples by the trio of cells $(x,y,z)$, in the orientation described. Let $a=\sigma(x)$, $b=\sigma(y)$, and $c=\sigma(z)$ be the contents of the cells of an $L$-triple, so the triple is in one of the configurations below: 
\begin{center}
\qtrip{$a$}{$b$}{$c$}, \qquad or \qquad \begin{tikzpicture}[scale=0.5]
\node at (-0.5,-.5) {$\emptyset$};
 \cell{2}{0}{$b$} \cell{2}{2.7}{$c$}
\node at (1,-1.5) {$\cdots$};
\end{tikzpicture}
\end{center}

We say that a triple $(x,y,z)$ is a \emph{queue inversion (quinv) triple} if its entries are oriented counterclockwise when the entries are read in increasing order, with ties being broken with respect to top-to-bottom, right-to-left reading order: in other words, if
\[a\leq b<c, \text{ or } b<c<a, \text{ or } c<a\leq b.\]
If the triple is degenerate with content $b,c$, it is a $\quinv$ triple if and only if $b<c$. 
\end{defn}

 \begin{figure}[ht!]
 \begin{tikzpicture}[scale=0.5]
\cell013\cell022
\cell111\cell123\cell131\cell143\cell153
\cell211\cell221\cell232\cell241\cell252\cell263\cell273
 \end{tikzpicture}
  \centering
  \caption{A tableau of type $\lambda=(3,3,2,2,2,1,1)$ and $n=3$. The weight of this filling is $x_1^5x_2^3x_3^6  q^5 t^{12}$.}\label{fig:tab}
 \end{figure}

The \emph{weight} of a filling $\sigma$ is then $x^{\sigma}t^{\quinv(\sigma)}q^{\maj(\sigma)}$, where $\quinv(\sigma)$ is the total number of $\quinv$ triples in $\sigma$, and $\maj$ is the classical major index statistic. See \cref{fig:tab} for an example of a tableau and its weight. Then we have the following.

\begin{theorem}[{\cite[Theorem 2.6]{AMM20}}]
\label{thm:mainconj}
Let $\lambda$ be a partition. The modified Macdonald polynomial can be written as
\begin{equation}\label{eq:H quinv}
\HH_{\lambda}(X;q,t) = \sum_{\sigma\in\Tab(\lambda)} x^{\sigma}t^{\quinv(\sigma)}q^{\maj(\sigma)}.
\end{equation}
\end{theorem}
The proof of this result in \cite{AMM20} is in the spirit of the original Haglund--Haiman--Loehr proof of \eqref{eq:H} in \cite{HHL05}: using LLT polynomials and super-symmetric fillings to show the formula satisfies the axioms that characterize the modified Macdonald polynomials.  In \cref{sec:plethysm}, modulo some technical details, we sketch a separate proof via superization by comparing tableaux with the quinv statistic to multiline queues.

\begin{remark} Although \eqref{eq:H quinv} appears very similar to \eqref{eq:HHL}, with the statistics $\inv$ and $\quinv$ being based on triples in the seemingly symmetric configurations \raisebox{-10pt}{\begin{tikzpicture}[scale=0.5]
\node at (0,0) {\tableau{x&&z\\y}};
\end{tikzpicture}} 
 and \raisebox{-10pt}{\begin{tikzpicture}[scale=0.5]
\node at (0,0) {\tableau{x\\y&&z}};
\end{tikzpicture}} 
 respectively, finding an explicit bijection between the two remains an open problem.\footnote{Since the writing of this article, Loehr gave bijective proofs of certain identities relating $\inv$ to $\quinv$ in \cite{Loehr-2022}. Even more recently, in \cite{JL24}, Jin and Lin gave a bijection on fillings $\Tab(\lambda)$ that maps $\inv$ to $\quinv$, thus resolving the question.} 
\end{remark}

\subsection{Queue inversion tableaux and probabilities of the mTAZRP}
The intended application for the fillings $\Tab(\lambda,n)$ with the $\quinv$ weight was originally to obtain formulas for probabilities of the mTAZRP of size $(\lambda,n)$. Recall the projection map $\proj:\MLQ(\lambda,n)\rightarrow \ASEP(\lambda,n)$ from \cref{def:MLQ}. We generalize this map to a map $\proj:\Tab(\lambda,n) \rightarrow \T(\lambda,n)$, to which we give the same name. (Implicitly, we are passing through the queue-tableau map $\mathcal{q}$.)

Let $\sigma\in\Tab(\lambda,n)$. Define $\proj(\sigma) = (w_1,\ldots,w_n)\in \TAZRP(\lambda,n)$, where $w_j = \{\lambda_s:1 \leq s \leq \ell(\lambda), \sigma(1,s)=j\}$. In other words, each site $w_j$ contains the multiset of column lengths of $\dg(\lambda)$ such that the entry in row 1 of those columns is equal to $j$. 

\begin{example} For the filling in \cref{fig:tab}, $\proj(\sigma)=(332|22|11)\in\T(3^22^31^3,3)$. 
\end{example}

Now, we obtain a formula for the stationary probabilities of the mTAZRP via this projection map, as an analog to \cref{thm:main mlq}.
\begin{theorem}[{\cite[Theorem 1.1]{AMM22}}]
Let $w\in\TAZRP(\lambda,n)$ be a state of the mTAZRP. Its stationary probability is given by
\[
\Pr(w) = \frac{1}{\widetilde{\mathcal{Z}}_{\lambda,n}} \sum_{\substack{\sigma\in\Tab(\lambda,n)\\\proj(\sigma)=w}}x^{\sigma}t^{\quinv(\sigma)},
\]
where $\widetilde{\mathcal{Z}}_{\lambda,n}$ is the combinatorial partition function of the mTAZRP:
\[
\widetilde{\mathcal{Z}}_{\lambda,n} = \sum_{\sigma\in\Tab(\lambda,n)}x^{\sigma}t^{\quinv(\sigma)}.
\]
\end{theorem}

This result is proved in \cite{AMM22} by explicitly constructing a Markov chain on $\Tab(\lambda,n)$ whose stationary distribution is given by $\wt(\sigma)(q=1)=x^{\sigma}t^{\quinv(\sigma)}$, and which lumps to (projects onto) the mTAZRP Markov chain via the projection map.

\begin{example} We show all six tableaux and their weights $x^{\sigma}t^{\quinv(\sigma)}$ corresponding to $w=(\cdot|21|1) \in\TAZRP((2,1,1),3)$:

\begin{center}
\begin{tikzpicture}[scale=0.4]
\cell00{$1$}
\cell10{$2$}\cell11{$3$}\cell12{$2$}
\node at (.5,-2) {$x_1x_2^2x_3t$};

\cell05{$2$}
\cell15{$2$}\cell16{$3$}\cell17{$2$}
\node at (5.5,-2) {$x_2^3x_3t$};

\cell0{10}{$3$}
\cell1{10}{$2$}\cell1{11}{$3$}\cell1{12}{$2$}
\node at (10.5,-2) {$x_2^2x_3^2$};

\cell0{15}{$1$}
\cell1{15}{$2$}\cell1{16}{$2$}\cell1{17}{$3$}
\node at (15.5,-2) {$x_1x_2^2x_3t^2$};

\cell0{20}{$2$}
\cell1{20}{$2$}\cell1{21}{$2$}\cell1{22}{$3$}
\node at (20.5,-2) {$x_2^3x_3t^2$};

\cell0{25}{$3$}
\cell1{25}{$2$}\cell1{26}{$2$}\cell1{27}{$3$}
\node at (25.5,-2) {$x_2^2x_3^2t$};
\end{tikzpicture}
\end{center}
Thus $\widetilde{\mathcal{Z}}_{(2,1,1),3}\Pr(w)=x_2^2x_3(1+t)(x_1t+x_2t+x_3)$.
\end{example}

\subsection{A compact $\quinv$ formula for $\HH_{\lambda}$}\label{sec:CHMMW20}

In \cite{CHMMW20}, a compact version of the Haglund--Haiman--Loehr formula was given for $\HH_{\lambda}$ when $\lambda$ is not a strict partition, by summing over fillings with a canonical ordering on columns of equal height: 
\begin{theorem}[{\cite{CHMMW20}}]
The modified Macdonald polynomial can be written as
\begin{equation}\label{eq:HHL compact}
\HH_{\lambda}(X;q,t) = \sum_{\substack{\sigma\in\Tab(\lambda)\\\sigma\,\isort}} \perm(\sigma)x^{\sigma}q^{\maj(\sigma)}t^{\inv(\sigma)},
\end{equation}
where a filling is \emph{inv sorted} if its columns of equal height are arranged in the unique order such that the filling has minimal $\inv$. $\perm(\sigma)$ is defined in \cref{def:perm}.
\end{theorem}

It turns out that an analogous compact formula exists for \cref{thm:mainconj}.

\begin{defn}\label{def:ordered}
We say $\sigma\in\Tab(\lambda,n)$ is \emph{quinv sorted} if for any pair of columns of equal height, the topmost pair of differing entries is not contained in a queue inversion. More precisely, for a pair of unequal columns $A,B$ of equal height $\ell$ with entries $a_{\ell},\ldots,a_1$ and $b_{\ell},\ldots,b_1$ respectively from top to bottom, let $1\leq r\leq \ell$ be maximal such that $a_r\neq b_r$. Then, we say $A>B$ if the cells corresponding to $a_{r+1},a_r,b_r$ form a quinv triple (as before, if $r=\ell$, by our convention we set $a_{r+1}=0$). A filling is \emph{quinv-sorted} if for every pair of columns of equal height $A,B$, if $A<B$, then $A$ is to the left of $B$. See \cref{ex:order} for the ordering and \cref{ex:quinv order} for a $\quinv$-sorted filling.
\end{defn}

\begin{example}\label{ex:order}
The $\quinv$-ordering for the following set of columns is:
\[
\raisebox{8pt}{\tableau{4\\2\\3}}\quad<\quad\raisebox{8pt}{\tableau{4\\1\\2}}\quad<\quad\raisebox{8pt}{\tableau{4\\1\\1}}\quad<\quad\raisebox{8pt}{\tableau{3\\2\\3}}\quad<\quad\raisebox{8pt}{\tableau{2\\2\\3}}\quad<\quad\raisebox{8pt}{\tableau{1\\3\\2}}
\]
\end{example}

Define the $t$-integer $[k]_t=\frac{1-t^k}{1-t}=1+t+\cdots+t^{k-1}$, the $t$-factorial
\[
[m]_t!=[m]_t[m-1]_t\cdots[2]_t,
\] 
and the $t$-multinomial coefficient for $m=m_1+\cdots+m_k$:
\[
{n\brack m_1,\ldots,m_k}=\frac{[n]_t!}{[m_1]_t![m_2]_t!\cdots[m_k]_t!}.
\]

\begin{defn}\label{def:perm}
Let $\sigma^{(i)}$ be a filling of a $m\times i$ rectangular diagram ($m$ columns of length $i$). Suppose $\sigma^{(i)}$ has $k$ distinct columns with multiplicities $m_1,\ldots,m_k$ (whence $m=m_1+\cdots+m_k$). We define the statistic $\perm(\sigma^{(i)})$ to be the inversion generating function of a word in the letters $\langle 1^{m_1}\cdots k^{m_k}\rangle$ (in frequency notation):
\[
\perm(\sigma^{(i)}) = {m \brack m_1,\ldots,m_k}_t.
\]
Then, for a filling $\sigma\in\Tab(\lambda,n)$, let $\sigma^{(i)}$ be the (possibly empty) rectangular block consisting of the columns of height $i$ of $\sigma$. Define
\[
\perm(\sigma) = \prod_{i=1}^{\lambda_1} \perm(\sigma^{(i)}).
\]
Described in another way, $\perm(\sigma)$ is simply the $t$-analog of the number of distinct ways of permuting the columns of $\sigma$ within the shape $\dg(\lambda)$.

We note that if all columns of $\sigma$ are distinct, then $\perm(\sigma)$ is equal to $\perm(\lambda)$ as defined in \cref{sec:J}.
\end{defn}

\begin{example}
Consider the tableau $\sigma$ in \cref{fig:tab}. $\perm(\sigma^{(1)})=1$ since the two columns of height 1 are equal, $\perm(\sigma^{(2)})=[3]_t[2]_t$, and $\perm(\sigma^{(1)})=[2]_t$. Thus $\perm(\sigma) = (1+t+t^2)(1+t)^2$.
\end{example}

\begin{theorem}\label{thm:compact}
The modified Macdonald polynomial can be written as
\begin{equation}\label{eq:H perm}
\HH_{\lambda}(X;q,t) = \sum_{\substack{\sigma\in\Tab(\lambda)\\\sigma\,\qsort}} \perm(\sigma)x^{\sigma}q^{\maj(\sigma)}t^{\quinv(\sigma)}.
\end{equation}
\end{theorem}

\begin{example}
As an example, we compute the following coefficient using six fillings with \eqref{eq:H perm} instead of the ten needed with \eqref{eq:H quinv}. 
\[[x_1^3x_2^2]H_{(2,2,1)}(X;q,t)=1+t+q+2t^2+2qt+q^2+t^3+qt^2.\]
\begin{center}
\renewcommand{\arraystretch}{2}
\begin{tabular}{c c c c c c}
\tableau{1&1\\2&2&1} &\tableau{1&1\\2&1&2} &\tableau{2&1\\1&2&1} &\tableau{2&1\\1&1&2} &\tableau{2&1\\2&1&1}&\tableau{2&2\\1&1&1}\\ 
$1$&$t(1+t)$&$q(1+t)$&$qt(1+t)$&$t^2(1+t)$&$q^2$
\end{tabular}
\end{center}
\end{example}

The proof of \cref{thm:compact} follows the strategy of \cite{CHMMW20}. We introduce entry-swapping operators $\{\tau_j\}$ to generate the entire set $\Tab(\lambda,n)$ from the set of $\quinv$-sorted tableaux, so that each sorted tableau $\sigma$ generates a set of fillings with total weight precisely equal to $\wt(\sigma)\perm(\sigma)$. 

\begin{defn}[The operators $\tau_j$]\label{def:tau}
Let $\sigma\in\Tab(\lambda,n)$, and suppose $\lambda_j=\lambda_{j+1}=k$. We will define an operator $\tau_j$ which exchanges contents of certain cells between columns $j$ and $j+1$. If the fillings of columns $j$ and $j+1$ are equal in $\sigma$, then $\tau_j(\sigma)=\sigma$. Otherwise, write $\sigma(r,j)=a_r$ and $\sigma(r,j+1)=b_r$ for $r=1,\dots, k$, and let $\ell$ be maximal such that $a_{\ell}\neq b_{\ell}$. $\tau_j$ swaps the pair $a_{\ell}$ and $b_{\ell}$ in columns $j,j+1$ in row $\ell$, and iteratively, if the pair in columns $j,j+1$ in row $i\leq \ell$ is swapped and this makes a difference to whether the triple $\{(i,j), (i-1, j), (i-1, j+1)\}$ is a quinv triple, then $\tau_j$ swaps the pair in columns $j,j+1$ of row $i-1$ also. The swapping terminates either when $i=1$, or when the swap at row $i$ doesn't change whether the triple $((i,j),(i-1,j),(i-1,j+1))$ is a quinv triple. Then $\tau_j(\sigma)$ is the tableau obtained by performing these iterative swaps in columns $j,j+1$, while keeping all other entries unchanged. See the diagram below, where the swapping terminates at row $h$. See \cref{ex:tau}.
\[
\begin{tikzpicture}[scale=.75]
\node at (0,1.5) {$\sigma$};
\node at (6,1.5) {$\tau_j(\sigma)$};
\node at (-2,.5) {\small row $k$};
\node at (-2,-.5) {$\vdots$};
\node at (-2,-2.5) {\small row $\ell$};
\node at (-2,-3.5) {$\vdots$};
\node at (-2,-4.5) {\small row $h$};
\cell00{$a_k$}\cell01{$a_k$}
\cell10{$\vdots$}\cell11{$\vdots$}
\cell20{$a_{\ell+1}$}\cell21{$a_{\ell+1}$}
\cell30{$a_{\ell}$}\cell31{$b_{\ell}$}
\cell40{$\vdots$}\cell41{$\vdots$}
\cell50{$a_h$}\cell51{$b_h$}
\cell60{$a_{h-1}$}\cell61{$b_{h-1}$}
\cell70{$\vdots$}\cell71{$\vdots$}
\node at (-0.5,-7.5) {\tiny $j$};
\node at (0.5,-7.5) {\tiny $j+1$};
\node at (3,-2.5) {$\xrightarrow{\makebox[1.5cm]{$\tau_j$}}$};
\cell06{$a_k$}\cell07{$a_k$}
\cell16{$\vdots$}\cell17{$\vdots$}
\cell26{$a_{\ell+1}$}\cell27{$a_{\ell+1}$}
\cell36{$b_{\ell}$}\cell37{$a_{\ell}$}
\cell46{$\vdots$}\cell47{$\vdots$}
\cell56{$b_h$}\cell57{$a_h$}
\cell66{$a_{h-1}$}\cell67{$b_{h-1}$}
\cell76{$\vdots$}\cell77{$\vdots$}
\node at (5.5,-7.5) {\tiny $j$};
\node at (6.5,-7.5) {\tiny $j+1$};
\end{tikzpicture}
\]
\end{defn}

\begin{example}\label{ex:quinv order}
For the tableau $\sigma$ below, we apply the sequence of operators $\tau_1, \tau_2, \tau_3$ to obtain the corresponding quinv-sorted tableau $\tau_3\circ\tau_2\circ\tau_1(\sigma)$:
\[
\sigma=\raisebox{8pt}{\tableau{1&3&2&1\\2&3&3&3&4\\2&1&4&1&2\\}}\xrightarrow{\makebox[.5cm]{$\tau_1$}} \raisebox{8pt}{\tableau{3&1&2&1\\3&2&3&3&4\\1&2&4&1&2\\}}\xrightarrow{\makebox[.5cm]{$\tau_2$}}  \raisebox{8pt}{\tableau{3&2&1&1\\3&2&3&3&4\\1&2&4&1&2\\}}\xrightarrow{\makebox[.5cm]{$\tau_3$}} \raisebox{8pt}{\tableau{3&2&1&1\\3&2&3&3&4\\1&2&1&4&2\\}}
\]
\end{example}

\begin{example}\label{ex:tau}
Suppose $\sigma$ has columns $j,j+1$ as shown below. Then $k=6$, $\ell=5$, and $h=3$, since swapping the 2 with the 3 at row 3 changes the triple at rows $h,h-1$ from $\tableau{2\\3&4}$ to $\tableau{3\\3&4}$, which is a quinv triple in both cases. Thus applying the operator $\tau_j$ gives the following.

\[
\begin{tikzpicture}[scale=.5]
\node at (0,1.5) {$\sigma$};
\node at (6,1.5) {$\tau_j(\sigma)$};
\node at (-2.5,.5) {\small row $k$};
\node at (-2.5,-.5) {\small row $\ell$};
\node at (-2.5,-1.5) {$\vdots$};
\node at (-2.5,-2.5) {\small row $h$};
\node at (-2.5,-3.5) {$\vdots$};
\node at (-2.5,-4.5) {\small row $1$};
\cell00{2}\cell01{2}
\cell10{3}\cell11{4}
\cell20{2}\cell21{3}
\cell30{2}\cell31{3}
\cell40{3}\cell41{4}
\cell50{1}\cell51{3}
\node at (-0.5,-5.5) {\tiny $j$};
\node at (0.5,-5.5) {\tiny $j+1$};
\node at (3,-1.5) {$\xrightarrow{\makebox[1.5cm]{$\tau_j$}}$};
\cell06{2}\cell07{2}
\cell16{4}\cell17{3}
\cell26{3}\cell27{2}
\cell36{3}\cell37{2}
\cell46{3}\cell47{4}
\cell56{1}\cell57{3}
\node at (5.5,-5.5) {\tiny $j$};
\node at (6.5,-5.5) {\tiny $j+1$};
\end{tikzpicture}
\]
\end{example}

The main function of the $\tau_j$'s is to swap entries so as to preserve the $\maj$ and change the $\quinv$ in a controlled way, as follows.

\begin{lemma}[{\cite[Lemmas 7.5 and 7.6]{AMM20}}]\label{thm:tau}
Let $\sigma\in\Tab(\lambda,n)$ with $\lambda_j=\lambda_{j+1}$.
\begin{itemize}
    \item If columns $j,j+1$ of $\sigma$ are not identical, then $\quinv(\tau_j(\sigma))=\quinv(\sigma)\pm 1$.
    \item $\maj(\tau_j(\sigma))= \maj(\sigma)$.
    \end{itemize}
\end{lemma}

\begin{proof}[Proof of \cref{thm:compact}]
In \cite{CHMMW20}, an analogous ordering on columns (we shall refer to it as an \emph{inv-order}) and an analogous set of operators $\{\widehat{\tau}_j\}$ are defined to be compatible with the $\inv$ statistic. They prove that applying the $\widehat{\tau}_j$'s to the set of inv-sorted tableaux generates $\Tab(\lambda,n)$, obtaining \eqref{eq:HHL compact}, which is a compact version of \eqref{eq:HHL}. Precisely the same arguments will prove \eqref{eq:H perm} for the $\tau_j$'s applied to quinv-sorted tableaux of \cref{def:ordered}. 

Unfortunately, the $\tau_j$'s do not in general satisfy braid relations. Thus we need to choose a canonical way to compose sequences of $\tau_j$'s, which for the $\widehat{\tau}_j$'s is done using \emph{positive distinguished subexpressions} (PDS) in \cite[\S 3.2]{CHMMW20}. Let $\sigma$ be a quinv-sorted filling. We apply the same reasoning here to define a family of tableaux $\{\tau_j\}\sigma$, generated by applying sequences of $\tau_j$'s obtained from a suitable set of PDS to $\sigma$. We omit all details here, since the arguments in \cite[\S 3.3]{CHMMW20} work just as well for our setting. Then,
\begin{equation}\label{eq:decomp}
\Tab(\lambda,n) = \biguplus_{\substack{\sigma\in\Tab(\lambda,n)\\\sigma\ \isort}} \{\tau_j\}\sigma.
\end{equation}

Let $\quinv'=\tau_{i_k}\circ\cdots\circ\tau_{i_1}(\sigma)\in\{\tau_j\}\sigma$, where $s_{i_k}\cdots s_{i_1}$ is a PDS. By \cref{thm:tau}, we deduce that $\quinv(\sigma')=t^k\quinv(\sigma)$. (Morally, this is permuting the columns of $\sigma$ by the permutation $s_{i_k}\cdots s_{i_1}$.) The length (i.e.~inversion) generating function of the set of PDS compatible with $\sigma$ is precisely equal to $\perm(\sigma)$. Thus we obtain
\[\sum_{\sigma'\in\{\tau_j\}\sigma}x^{\sigma'}q^{\maj(\sigma')}t^{\quinv(\sigma')} = x^{\sigma}q^{\maj(\sigma)}t^{\quinv(\sigma)}\perm(\sigma).\]
Combined with \eqref{eq:decomp}, this proves the result.
\end{proof}

\section{On the plethystic relationship between the $\quinv$ tableau formula and multiline queues}\label{sec:plethysm}

We begin with a discussion of the modified Macdonald polynomials $\HH_{\lambda}(X;q,t)$ and the intuition behind the combinatorial interpretation of the plethystic substitution leading to the introduction of the \emph{queue inversion} (quinv) statistic and the discovery of the connection with the mTAZRP. Recall the definition of $\HH_{\lambda}$ in \eqref{eq:H} and the multiline queue formula \eqref{eq:J} for $J_{\lambda}(X;q,t)$. For a symmetric function $f$, the plethystic notation $f[X/(1-t^1)]$ represents the substitution $p_k(x)\rightarrow p_k(x)/(1-t^k)$ into the power-sum expansion of $f$ (where $p_k(x)=\sum_i x^k$ is the $k$'th power sum symmetric function). In \eqref{eq:H}, this translates to the formal substitution of the monomials $X=x_1,\ldots,x_n$ by the monomials $X/(1-t^{-1})=x_1,x_1t^{-1},x_1t^{-2},\ldots,x_2,x_2t^{-1},x_2t^{-2},\ldots,x_n,x_nt^{-1},x_nt^{-2},\ldots$ into $J_{\lambda}(X;q,t^{-1})$. Expanding \eqref{eq:H} as
\[\HH_{\lambda}(X;q,t)=t^{n(\lambda)}J_{\lambda}(x_1,x_1t^{-1},x_1t^{-2},\ldots,x_2,x_2t^{-1},x_2t^{-2},\ldots;q,t^{-1})\]
and using \eqref{eq:J}, we may interpret this polynomial as a generating function over infinite multiline queues. Specifically, we consider multiline queues with a countable number of columns labeled by $x_1,x_1t^{-1},x_1t^{-2},\ldots,x_2,x_2t^{-1},x_2t^{-2},\ldots$, so that any ball in a column labeled by $x_it^{-j}$ contributes $x_it^{-j}$ to the weight of the configuration. To get back to a finite object, we group the infinite multiline queues into a finite set of equivalence classes, by ``forgetting" the powers of $t$ labeling each column while keeping the same pairings, and associating an appropriate weight to each pairing to account for the compression and the lost factors of $t^{-j}$. This results in a ``multi-capacity multiline queue'', which we define to be a \emph{multiline diagram} in which multiple balls can occupy any given site.\footnote{One could also call this a "fused multiline queue", or "bosonic multiline queue" following \cref{rem:fusion}.} 

Although it becomes cumbersome to evaluate the compression factor directly, this notion of a multi-capacity queue led the authors of \cite{CHMMW20} to a new combinatorial formula for $\HH_{\lambda}(X;q,t)$ through the $\Skip$ statistic of multiline queues, which was reformulated as the queue inversion statistic (the formula was subsequently proved in \cite{AMM20}). To be precise, the queue inversion statistic is essentially the negation of the "$\Skip$" statistic in the multiline queues, and this relationship in many ways parallels the transformation from Haglund--Haiman--Loehr nonattacking fillings and coinversions for computing $P_{\lambda}$ and $J_{\lambda}$ to unrestricted fillings and inversions for computing $\HH_{\lambda}$. We note that \cite{AMM20} focused on tableaux rather than multiline diagrams, so obtaining analogous formulas in terms of multiline diagrams is yet to be done in a rigorous manner. 

\begin{remark}\label{rem:fusion}
In fact, it was shown in \cite{garbali-wheeler-2020} that the plethystic substitution to obtain $\HH_{\lambda}$ from $J_{\lambda}$ corresponds to a known concept in physics called \emph{fusion}, introduced by Kulish--Reshetikhin--Sklyanin \cite{KRS81}. In physics, a \emph{fermionic} particle system is one having the property of ``exclusion", where any site can contain at most one particle, whereas a \emph{bosonic} particle system is one without this restriction. With this terminology, the ASEP is fermionic and the TAZRP is bosonic. The compression we describe above of an infinite family of fermionic objects (multiline queues with countably many columns for $P_{\lambda}$) into a finite family of bosonic objects (multiline diagrams for $\HH_{\lambda}$) is an example of fusion. Using matrix products, \cite{garbali-wheeler-2020} express the fused weights of the fermionic lattice model for $J_{\lambda}$ from \cite{cantini-etal-2016} as infinite sums that end up being positive polynomials in all parameters. We have not found a way to explicitly connect our results to theirs, but it would be very interesting to understand this relationship.
\end{remark}
 
 \subsection{Using superization to obtain a new formula for $P_{\lambda}$ from \eqref{eq:H quinv}}
In \cite[\S 8]{HHL05} a formula for $J_{\lambda}(X;q,t)$ is derived from the extension of \eqref{eq:HHL} to super-fillings. Their formula reduces to the formula \cite[Proposition 8.1]{HHL05} for $J_{\lambda}(X;q,t)$ as a sum over non-attacking fillings of $\dg(\lambda)$ (see \cref{def:attacking}) using the coinversion statistic. We use the same strategy here to derive a quinv analog, which will correspond to the ``compression factor'' described above.

For a given $n\in\mathbb{N}$, define the alphabet $\mathcal{A} = \{1,\bar{1},2,\bar{2},\ldots,n,\bar{n}\}$. A \emph{super-filling} of $\dg(\lambda)$ is a function $\sigma:\dg(\lambda)\rightarrow \mathcal{A}$, and is denoted by $\widetilde{\Tab}(\lambda,n)$. For $\sigma\in\widetilde{\Tab}(\lambda,n)$, define $p(\sigma)$ and $m(\sigma)$ to be the numbers of positive and negative entries in $\sigma$, respectively, and define $|\sigma|\in\Tab(\lambda,n)$ to be the filling whose entries are the absolute values of the entries in $\sigma$. Fix the total ordering $<$ on $\mathcal{A}\cup\{0\}$ to be  $\{0<1<\bar{1}<2<\bar{2}<\cdots< n<\bar{n}\}$. For $a,b\in\mathcal{A}\cup\{0\}$ we introduce the notation
\[
I(a,b)=\begin{cases} a> b,&a\neq b\\ 0&a\in\mathbb{Z}_+\\ 1&a\in\mathbb{Z}_-.
\end{cases}
\] 
to generalize the notion of ``greater than" on $\mathbb{Z}_+$. We can now extend the definition of a descent to a super-filling: $x\in\dg(\lambda)\in\Des(\sigma)$ if $I(\sigma(x),\sigma(\South(x))=1$. We also extend the definition of a quinv triple: an $L$-triple $(x,y,z)$ is a quinv triple if exactly \emph{one} of the following is true:
\[
\{I(\sigma(x),\sigma(y))=1,I(\sigma(z),\sigma(y))=0,I(\sigma(x),\sigma(z))=0\}.
\]
When $\sigma=|\sigma|$, we recover the original definitions of $\Des$ and $\quinv$.

Using \cite[Proposition 4.3]{HHL05}, the following tableaux formula is derived in \cite[\S 3]{AMM20}:
\begin{equation}\label{eq:C}
\HH_{\lambda}[X(t-1);q,t] = \sum_{\sigma\in\widetilde{\Tab}(\lambda,n)}(-1)^{m(\sigma)}x^{|\sigma|}q^{\maj(\sigma)}t^{p(\sigma)+\quinv(\sigma)}.
\end{equation}

Rearranging \eqref{eq:H}, one obtains an expression for $J_{\lambda}(X;q,t)$:
\begin{align}
\HH_{\lambda}(X;q,t^{-1})&=t^{-n(\lambda)}J_{\lambda}\Big[\frac{X}{1-t};q,t\Big],\nonumber\\
t^{n(\lambda)}\HH_{\lambda}[X(1-t);q,t^{-1}]&=J_{\lambda}\Big[X;q,t\Big],\nonumber\\
t^{n(\lambda)+n}\HH_{\lambda}[X(t^{-1}-1);q,t^{-1}]&=J_{\lambda}(X;q,t).\label{eq:J}
\end{align}

We invoke \cite[Theorem 5.3]{AMM20} to write \eqref{eq:C} as a sum over quinv-non-attacking fillings (see \cref{def:attacking}):
\begin{equation}\label{eq:H nonattacking}
J_{\lambda}(X;q,t) = t^{n(\lambda)+n} \sum_{\substack{\sigma\in\widetilde{\Tab}(\lambda,n)\\|\sigma|\,\qnona}}(-1)^{m(\sigma)}x^{|\sigma|}q^{\maj(\sigma)}t^{-p(\sigma)-\quinv(\sigma)}.
\end{equation}

To reduce \eqref{eq:H nonattacking} to quinv-non-attacking fillings of $\Tab(\lambda,n)$ in the positive alphabet, we follow the strategy in \cite[\S 8]{HHL05} . For each $\tau\in\Tab(\lambda,n)$, we consider the set of fillings $\sigma\in\widetilde{\Tab}(\lambda,n)$ such that $\tau=|\sigma|$. We observe that for any entries $a,b\in \mathcal{A}$, if $|a|\neq|b|$, then $I(a,b)=I(|a|,|b|)$. Moreover, in any $L$-triple $(x,y,z)$, either $\tau(x)\neq\tau(y)\neq\tau(z)$, or $\tau(x)=\tau(y)\neq\tau(z)$ (where $y=\South(x)$) by the non-attacking condition. So, to obtain the difference $\quinv(\sigma)-\quinv(\tau)$ and $\maj(\sigma)-\maj(\tau)$, we need only to consider descents and $L$-triples that involve a cell $u\in\dg(\lambda)$ such that $\tau(u)=\tau(\South(u))$.

\begin{defn}\label{def:rarm}
For a cell $u\in\dg(\lambda)$ in a row $r>1$, define the statistic $\rarm(u)$ to be the number of cells to the right of $u$ in the row below. In other words, $\rarm(r,c) = |\{(r-1,j)\in\dg(\lambda):j>c\}|$, which is also equal to the leg of $\South(u)$ in $\dg(\lambda')$. See \cref{fig:leg} for an example of $\rarm$. \end{defn}

Each $u\in\dg(\lambda)$ such that $\tau(u)\neq \tau(\South(u))$ brings a factor of $(t^{-1}-1)$ for $\sigma(u)\in\mathbb{Z}_+$ and $\sigma(u)\in\mathbb{Z}_-$, respectively (as $\quinv(\sigma)=\quinv(\tau)$ and $\maj(\sigma)=\maj(\tau)$ for such $u$). Now consider $u\in\dg(\lambda)$ such that $\tau(u)=\tau(\South(u))$.
\begin{itemize}
    \item If $\sigma(u)\in\mathbb{Z}_+$, then $I(\sigma(u),\sigma(\South(u)))=I(\tau(u),\tau(\South(u)))=0$ and so $\maj(\sigma)=\maj(\tau)$. Moreover, every $L$-triple $(u,\South(u),z)$ is a quinv triple in $\sigma$  and in $\tau$. Thus every such $u$ contributes a $t^{-1}$.
    \item If $\sigma(u)\in\mathbb{Z}_-$, then $I(\sigma(u),\sigma(\South(u)))=1-I(\tau(u),\tau(\South(u)))=1$ and thus $(u,\South(u),z)$ is never a quinv triple in $\sigma$ for any $z$, while it is one in $\tau$; the number of choices for the cell $z$ is equal to $\rarm(u)$. Moreover, this $u$ is a descent in $\sigma$, but not one in $\tau$.  Thus, every such $u$ contributes $(-\rarm(u))$ to the difference $\quinv(\sigma)-\quinv(\tau)$, and $\leg(u)+1$ to the difference $\maj(\sigma)-\maj(\tau)$.
\end{itemize}
We conclude that the total contribution for every $u$ such that $\tau(u)=\tau(\South(u))$ from all choices for $\sigma$ is $(t^{-1}-q^{\leg(u)+1}t^{\rarm(u)})$. 

Thus we obtain the following formula, after distributing the $t^n$ to the contribution for each cell $u\in\dg(\lambda)$, where \eqref{eq:P from H} follows from \eqref{eq:J}. Define $\coquinv(\sigma) := n(\lambda)-\quinv(\sigma)$. 

\begin{theorem}
Fix a partition $\lambda$. Then
\begin{multline}\label{eq:J from H}
J_{\lambda}(X;q,t) =  \sum_{\substack{\tau\in\Tab(\lambda)\\\tau\,\qnona}}x^{\tau}q^{\maj(\tau)}t^{\coquinv(\tau)}
 \prod_{\substack{u\in\dg(\lambda)\\\tau(u)\neq\tau(\South(u))}}(1-t)\\
\times \prod_{\substack{u\in\dg(\lambda)\\\tau(u)=\tau(\South(u))}}(1-q^{\leg(u)+1}t^{\rarm(u)+1}).
\end{multline}
and 
\begin{multline}\label{eq:P from H}
P_{\lambda}(X;q,t) =  \frac{1}{\perm(\lambda)}\sum_{\substack{\tau\in\Tab(\lambda)\\\tau\,\qnona}}x^{\tau}q^{\maj(\tau)}t^{\coquinv(\tau)}
\\
\times \prod_{\substack{u\in\dg(\lambda)\\\tau(u)\neq\tau(\South(u))}}\frac{1-t}{1-q^{\leg(u)+1}t^{\rarm(u)+1}},
\end{multline}
where both sums are over quinv-non-attacking tableaux.
\end{theorem}

We have ended up with formulas for $J_{\lambda}(X;q,t)$ and $P_{\lambda}(X;q,t)$ as sums over quinv-non-attacking tableaux in terms of the statistics $\coquinv$ and $\rarm$, a simplification from the more technical Haglund--Haiman--Loehr (HHL) formula which involves ``Type A'' and ``Type B'' triples (see \cite{HHL05} for the definitions). 

\begin{example}
We compute the coefficients of $m_{22}$ and $m_{211}$ for $P_{(2,2)}(X;q,t)$ using \eqref{eq:P from H} to get $m_{22}+\frac{(1+q)(1-t)}{1-qt}m_{211}$ after dividing by $\perm((2,2))=(1+t)$.
\begin{center}
\renewcommand{\arraystretch}{2}
\begin{tabular}{c c | c c c c c c}
\tableau{1&2\\1&2}& \tableau{2&1\\2&1}& \tableau{1&2\\1&3} &  \tableau{2&1\\3&1} & \tableau{1&3\\1&2}& \tableau{3&1\\2&1}& \tableau{3&1\\1&2}  & \tableau{2&1\\1&3}\\
$1$ & $t$ & $\frac{(1-t)}{1-qt}$ & $\frac{t(1-t)}{1-qt^2}$ & $\frac{q(1-t)}{1-qt}$ & $\frac{qt^2(1-t)}{1-qt^2}$ &$\frac{qt(1-t)^2}{(1-qt)(1-qt^2)}$ &$\frac{qt^2(1-t)^2}{(1-qt)(1-qt^2)}$
\end{tabular}
\end{center}
\end{example}

\begin{remark}
Notice the extra factor of $\perm(\lambda)$ in \eqref{eq:P from H} that is not present in \eqref{eq:P mlq}. This is due to the fact that there is no prescribed order on the columns of a filling in \eqref{eq:P from H}, the way there is in a multiline queue, but instead, each permutation of columns of the same height corresponds to a distinct tableau. Thus the formula \eqref{eq:P from H}, just like the original HHL formula for $P_{\lambda}$, enumerates $\perm(\lambda)(t=1)$ times more tableaux than the multiline queue formula \eqref{eq:P mlq}! Asymptotically, this is the same factor that compares the number of terms in the compact formulas \eqref{eq:H perm} and \eqref{eq:HHL compact} for $\HH_{\lambda}$ to the number of terms in the original HHL formula \eqref{eq:HHL}.
\end{remark}

In particular, we conjecture a \emph{new} compact formula for $P_{\lambda}$ summing over $\quinv$-sorted tableaux and using the $\coquinv$ statistic. 
\begin{conjecture}\label{conj:compact}
The Macdonald polynomial $P_{\lambda}(X;q,t)$ is given by
\begin{equation}\label{eq:P from H compact}
P_{\lambda}(X;q,t) =  \hspace{-0.2in} \sum_{\substack{\tau\in\Tab(\lambda)\\\tau\,\qnona\\
\tau\,\qsort}}\hspace{-0.2in} x^{\tau}q^{\maj(\tau)}t^{\coquinv(\tau)} \hspace{-0.2in}\prod_{\substack{u\in\dg(\lambda)\\\tau(u)\neq\tau(\South(u))}}\hspace{-0.2in}\frac{1-t}{1-q^{\leg(u)+1}t^{\rarm(u)+1}}.
\end{equation}
\end{conjecture}

\begin{example}
We compute $P_{(2,2)}(X;q,t)$ using \eqref{eq:P from H compact} to get
\[
P_{22}(X;q,t) = m_{22}+\frac{(1+q)(1-t)}{1-qt}m_{211}+ \frac{(2+t+3q+q^2+3qt+2q^2t)(1-t)^2}{(1-qt)(1-qt^2)}m_{1^4}.
\]
\begin{center}
\renewcommand{\arraystretch}{2}
\begin{tabular}{c | c c | c c c c c c }
\tableau{1&2\\1&2}& \tableau{1&2\\1&3} & \tableau{1&3\\1&2} &  \tableau{1&2\\3&4}& \tableau{1&2\\4&3}  & \tableau{1&3\\2&4}  & \tableau{1&3\\4&2}  & \tableau{1&4\\2&3}  & \tableau{1&4\\3&2}\\
$1$ & $\frac{(1-t)}{1-qt}$ & $\frac{q(1-t)}{1-qt}$ & $1$ & $t$ & $1$ & $qt$ & $q$ & $qt$\\
 &&&\tableau{2&3\\1&4}  & \tableau{2&3\\4&1}  & \tableau{2&4\\1&3}  & \tableau{2&4\\3&1}  & \tableau{3&4\\1&2}  & \tableau{3&4\\2&1}\\
  & && $qt$ & $q$ & $q^2t$ & $q$ & $q^2$ & $q^2t$
\end{tabular}
\end{center}
(To avoid clutter, the common factor of $\frac{(1-t)^2}{(1-qt)(1-qt^2)}$ was left out from the weights of the fillings for $m_{1111}$.)
\end{example}

One aesthetically pleasing property of this formula is its symmetry. The orientation of $L$-triples agrees with the orientation of the Ferrers diagram. Moreover, for each cell $u\in\dg(\lambda)$ in a row $r>1$, the sum $\leg(u)+1+\rarm(u)+1$ is equal to the hook length of $\South(u)$. This symmetry is curious in light of the mysterious $q,t$-symmetry of the Macdonald polynomial. 

We also note that this formula coincides with Lenart's formula for $P_{\lambda}$ in \cite{Lenart} (as well as with \eqref{eq:P mlq}) when $\lambda$ is a strict partition (has no repeating parts). 

\begin{remark}
During the preparation of this article for publication, the author proved the conjecture in \cite{Man24} using probabilistic column swapping operators that generalize the $\tau_i$'s defined in \cref{def:tau}.  
\end{remark}

\subsection{Comparison to multiline queues}\label{sec:comparison}
Observe that quinv-sorted non-attacking tableaux appear remarkably similar to the multiline queues defined in \cite{CMW18} after passing through the queue-tableau map $\mathcal{q}$. In fact, the difference comes exclusively from our rule about trivial pairings: within each row of a multiline queue, a pairing order as defined in \cref{def:pairing} is required to place trivial pairings first within each class of labels. Thus trivially-paired balls cannot count for the $\Skip$ or $\Free$ statistics. Moreover, in a quinv non-attacking tableau, the same content can only appear in adjacent rows if it is in the same column, or in two columns of different height. On the other hand, in the multiline queues of Martin in \cite{martin-2020}, the trivial-pairing condition is replaced by the condition ``if there is a free ball directly below a departure site, then the pairing is forced to be trivial". In fact, Martin's multiline queues in \cite{martin-2020} are in bijection with quinv-non-attacking quinv sorted tableaux via the queue-tableau map $\mathcal{q}$, which also preserves the associated statistics.

\renewcommand{\MR}[1]{} 
\printbibliography
\end{document}